\documentclass[oneside,english]{amsart}
\usepackage[T1]{fontenc}
\usepackage[latin9]{inputenc}
\usepackage{geometry}
\usepackage{mathrsfs}
\usepackage{amsbsy}
\usepackage{amstext}
\usepackage{amsthm}
\usepackage{amssymb}
\usepackage{graphicx}
\usepackage[all]{xy}

\makeatletter
\numberwithin{equation}{section}
\numberwithin{figure}{section}
\numberwithin{table}{section}
\theoremstyle{plain}
\newtheorem{thm}{\protect\theoremname}[section]
  \theoremstyle{definition}
  \newtheorem{defn}[thm]{\protect\definitionname}
  \theoremstyle{plain}
  \newtheorem{prop}[thm]{\protect\propositionname}
  \theoremstyle{remark}
  \newtheorem*{rem*}{\protect\remarkname}
  \theoremstyle{plain}
  \newtheorem{lem}[thm]{\protect\lemmaname}
  \theoremstyle{plain}
  \newtheorem{cor}[thm]{\protect\corollaryname}
  \theoremstyle{remark}
  \newtheorem{rem}[thm]{\protect\remarkname}

\author{Sungkyung Kang}
\address{Mathematical Institute, University of Oxford, Andrew Wiles Building,
Radcliffe Observatory Quarter, Woodstock Road, Oxford, OX2 6GG, UK}
\email{sungkyung.kang@maths.ox.ac.uk}
\thanks{This project has received funding from the European Research Council (ERC) under the European Unions Horizon 2020 research and innovation programme (grant agreement No 674978).}


\makeatother

\usepackage{babel}
  \providecommand{\corollaryname}{Corollary}
  \providecommand{\definitionname}{Definition}
  \providecommand{\lemmaname}{Lemma}
  \providecommand{\propositionname}{Proposition}
  \providecommand{\remarkname}{Remark}
\providecommand{\theoremname}{Theorem}

\begin{document}

\title{$\mathbb{Z}_{2}$-equivariant Heegaard Floer cohomology of knots
in $S^{3}$ as a strong Heegaard invariant}
\begin{abstract}
The $\mathbb{Z}_{2}$-equivariant Heegaard Floer cohomlogy $\widehat{HF}_{\mathbb{Z}_{2}}(\Sigma(K))$
of a knot $K$ in $S^{3}$, constructed by Hendricks, Lipshitz, and
Sarkar, is an isotopy invariant which is defined using bridge diagrams
of $K$ drawn on a sphere. We prove that $\widehat{HF}_{\mathbb{Z}_{2}}(\Sigma(K))$
can be computed from knot Heegaard diagrams of $K$ and show that
it is a strong Heegaard invariant. As a topolocial application, we
construct a transverse knot invariant $\hat{\mathcal{T}}_{\mathbb{Z}_{2}}(K)$
as an element of $\widehat{HFK}_{\mathbb{Z}_{2}}(\Sigma(K),K)$, which
is a refinement of $\widehat{HF}_{\mathbb{Z}_{2}}(\Sigma(K))$, and
show that it is a refinement of both the LOSS invariant $\hat{\mathcal{T}}(K)$
and the $\mathbb{Z}_{2}$-equivariant contact class $c_{\mathbb{Z}_{2}}(\xi_{K})$.
\end{abstract}

\maketitle

\section{Introduction}

Suppose that a based knot $K$ in $S^{3}$ is given. Then we can represent
$K$ as a bridge diagram on a sphere, and taking its branched double
cover along the points where $K$ and the sphere intersect gives a
Heegaard diagram of the branched double cover $\Sigma(K)$ of $S^{3}$
along $K$. This diagram admits a natural $\mathbb{Z}_{2}$-action
which fixes the basepoint and the $\alpha$,$\beta$-curves. From
these data, Hendricks, Lipshitz, and Sarkar\cite{eqv-Floer} gave
a construction of the $\mathbb{Z}_{2}$-equivariant Heegaard Floer
cohomology $\widehat{HF}_{\mathbb{Z}_{2}}(\Sigma(K))$, using their
formulation of $\mathbb{Z}_{2}$-equivariant Floer cohomology theory.
They also proved that the isomorphism class of $\widehat{HF}_{\mathbb{Z}_{2}}(\Sigma(K))$,
which is a $\mathbb{F}_{2}[\theta]$-module where $\theta$ is a formal
variable, is an invariant of the isotopy class of the given knot $K$.
Also, the author proved in \cite{Kang} that $\widehat{HF}_{\mathbb{Z}_{2}}(\Sigma(K))$
satisfies naturality and is functorial under based link cobordisms
whose ends are knots. 

Given these facts, it is natural to ask whether $\widehat{HF}_{\mathbb{Z}_{2}}(\Sigma(K))$
can be computed from bridge diagrams of $K$ drawn on closed surfaces
of aribtrary genera, instead of spheres. In section 2, we will see
that this is almost always possible, by proving the following theorem.
\begin{thm}
\label{thm:main1}Let $\mathcal{E}$ be a weakly admissible extended
bridge diagrams representing a knot in $S^{3}$, which has at least
two A-arcs. Then $\widehat{HF}_{\mathbb{Z}_{2}}(\mathcal{E})\simeq\widehat{HF}_{\mathbb{Z}_{2}}(\Sigma(K))$.
\end{thm}

Now, given such a fact, we can use it to compute $\widehat{HF}_{\mathbb{Z}_{2}}(\Sigma(K))$
from a weakly admissible knot Heegaard diagram of $K$. To write it
up clearly, choose a weakly admissible knot Heegaard diagram $H=(\Sigma,\boldsymbol{\alpha},\boldsymbol{\beta},z,w)$
which represents a based knot $(K,z)$. Add a pair of a small A-arc
and a small B-arc connected to $w$, whose interiors are disjoint;
this gives a weakly admissible extended bridge diagram representing
$K$, which has at least two A-arcs. Then taking the branched double
cover of the resulting extended bridge diagram and forgetting all
basepoints except $z$ gives a Heegaard diagram of $\Sigma(K)$ together
with a $\mathbb{Z}_{2}$-action. By Theorem \ref{thm:main1}, the
$\mathbb{Z}_{2}$-equivariant Heegaard Floer cohomology of this diagram
is isomorphic to $\widehat{HF}_{\mathbb{Z}_{2}}(\Sigma(K))$.

This construction implies that $\widehat{HF}_{\mathbb{Z}_{2}}(\Sigma(K))$
is a weak Heegaard invariant of $K$, as defined in \cite{Juhasz-naturality}.
In section 3, we will see that the $\mathbb{Z}_{2}$-equivariant Heegaard
Floer cohomology, as a weak Heegaard invariant, satisfies certain
commutativity axioms, thereby proving that it is actually a strong
Heegaard invariant. Morover, we will also see that $\widehat{HF}_{\mathbb{Z}_{2}}$
as a natural invariant calculated from bridge diagrams on a sphere
is naturally isomorphic to $\widehat{HF}_{\mathbb{Z}_{2}}$ as a strong
Heegaard invariant; to be precise, we will prove the following theorem.
\begin{thm}
\label{thm:main2}Consider the category~$\text{Knot}_{\ast}$ whose
objects are based knots in $S^{3}$ and morphisms are self-diffeomorphisms
of $S^{3}$, and let 
\[
\widehat{HF}_{\mathbb{Z}_{2}}^{bridge}\,:\,\text{Knot}_{\ast}\rightarrow\text{Mod}_{\mathbb{F}_{2}[\theta]}
\]
 be the functor defined in Theorem 6.9 of \cite{Kang}. Also, let
\[
\widehat{HF}_{\mathbb{Z}_{2}}^{knot}\,:\,\text{Knot}_{\ast}\rightarrow\text{Mod}_{\mathbb{F}_{2}[\theta]}
\]
 be the functor defined by considering $\widehat{HF}_{\mathbb{Z}_{2}}$
as a strong Heegaard invariant. Then there exists an invertible natural
transformation between $\widehat{HF}_{\mathbb{Z}_{2}}^{bridge}$ and
$\widehat{HF}_{\mathbb{Z}_{2}}$.
\end{thm}

In section 4, we will see that any knot Heegaard diagrams representing
a (based) knot $K$ in $S^{3}$ can be transformed, via isotopies
and handleslides, to certain types of knot Heegaard diagrams, called
very nice diagrams. Also, we will see that, from such a diagram, we
can compute $\widehat{HF}_{\mathbb{Z}_{2}}(\Sigma(K))$ in a purely
combinatorial way. As a result, we can remove a pair of an A-arc and
a B-arc when computing $\widehat{HF}_{\mathbb{Z}_{2}}$ from a knot
Heegaard diagram, and thus extend Theorem \ref{thm:main1} to full
generality.
\begin{thm}
\label{thm:main3}Let $\mathcal{E}$ be a weakly admissible extended
bridge diagrams representing a knot in $S^{3}$. Then $\widehat{HF}_{\mathbb{Z}_{2}}(\mathcal{E})\simeq\widehat{HF}_{\mathbb{Z}_{2}}(\Sigma(K))$.
\end{thm}

In section 5, we construct a new invariant $\widehat{HFK}_{\mathbb{Z}_{2}}(\Sigma(K),K)$
associated to a knot $K$ in $S^{3}$, whose isomorphism class is
also an invariant of the isotopy class of $K$, and prove that a version
of localization isomorphism exists for $\widehat{HFK}_{\mathbb{Z}_{2}}$.
Finally, in section 6, we will construct an element $\hat{\mathcal{T}}_{\mathbb{Z}_{2}}(K)\in\widehat{HFK}_{\mathbb{Z}_{2}}(\Sigma(K),K)$
associated to a transverse knot $K$ in the standard contact sphere
$(S^{3},\xi_{std})$, which depends only on the transverse isotopy
class of $K$, and see that it is a refinement of both the LOSS invariant
defined in \cite{LOSS-inv} and the $\mathbb{Z}_{2}$-equivariant
contact class defined by the author in \cite{Kang}.

\subsection*{Acknowledgement}

The author would like to thank Andras Juhasz and Robert Lipshitz for
helpful discussions and suggestions.

\section{Equivariant Heegaard Floer cohomology and extended bridge diagrams}
\begin{defn}
Suppose that a Heegaard diagram $\mathcal{H}=(\Sigma,\boldsymbol{\alpha},\boldsymbol{\beta})$
is given. A based bridge diagram on $\mathcal{H}$ is a 4-tuple $(F,A,B,z)$,
where $F\subset\Sigma$ is a finite subset of points in $\Sigma$,
$A,B$ are sets of simple arcs on $\Sigma$, and $z\in F$, such that
the following properties are satisfied.
\begin{itemize}
\item For any $\gamma\in\boldsymbol{\alpha}\cup\boldsymbol{\beta}$, we
have $F\cap\gamma=\emptyset$.
\item For any two distinct elements $\boldsymbol{a},\boldsymbol{a}^{\prime}\in A\cup\boldsymbol{\alpha}$,
we have $\boldsymbol{a}\cap\boldsymbol{a}^{\prime}=\emptyset$, and
the same statement holds for elements in $B\cup\boldsymbol{\beta}$.
\item For any $\boldsymbol{a}\in A\cup\boldsymbol{\alpha}$ and $\boldsymbol{b}\in B\cup\boldsymbol{\beta}$,
the intersection $\boldsymbol{a}\cap\boldsymbol{b}$ is transverse.
\item For any $\boldsymbol{c}\in A\cup B$, the two endpoints of $\boldsymbol{c}$
are distinct and $\boldsymbol{c}\cap F=\partial\boldsymbol{c}$.
\item For any $p\in F$, there exists a unique element $\boldsymbol{a}$
of $A$ which satisfies $p\in\partial\boldsymbol{a}$, and the same
statement holds for $B$.
\end{itemize}
\end{defn}

Given a pointed bridge diagram $(F,A,B,z)$ on a Heegaard diagram
$(\Sigma,\boldsymbol{\alpha},\boldsymbol{\beta})$, we call the elements
of $A$ as A-arcs, the elements of $B$ as B-arcs, and $z$ as the
basepoint. Note that $(\Sigma,\boldsymbol{\alpha},\boldsymbol{\beta},z)$
is a pointed Heegaard diagram.

Suppose that a Heegaard diagram $\mathcal{H}=(\Sigma,\boldsymbol{\alpha},\boldsymbol{\beta})$
describes a 3-manifold $M$. Then, given a based bridge diagram $\mathcal{P}=(F,A,B,z)$
on $\mathcal{H}$, we can construct a based link $(L,p)$ lying inside
$M$ as follows. Let $M=H_{1}\cup H_{2}$ be a Heegaard splitting
of $M$ given by the Heegaard surface $\Sigma$. Suppose that we call
$H_{1}$ as the ``outside'' of $\Sigma$ and $H_{2}$ as the ``inside''
of $\Sigma$. Then, we can isotope the A-arcs of $\mathcal{P}$ slightly
ourwards and the B-arcs of $\mathcal{P}$ slightly inwards, while
leaving the set $F$ fixed. Concatenating the isotoped arcs gives
us a link $L\subset M$, and the basepoint $p$ lies on $L$, so that
we get a based link $(L,p)$ in $M$, which is uniquely determined
up to (based) isotopy.
\begin{defn}
Suppose that a Heegaard diagram $\mathcal{H}=(\Sigma,\boldsymbol{\alpha},\boldsymbol{\beta})$
describes a 3-manifold $M$. We say that a based link $(L,p)$ in
$M$ is represented by a based bridge diagram $\mathcal{P}=(F,A,B,z)$
if the process described above gives a based link which is (based)
isotopic to $(L,p)$.
\end{defn}

\begin{prop}
Suppose that a Heegaard diagram $\mathcal{H}=(\Sigma,\boldsymbol{\alpha},\boldsymbol{\beta})$
describes a 3-manifold $M$. Then every based link in $M$ can be
represented by a based bridge diagram on $\mathcal{H}$. 
\end{prop}

\begin{proof}
Let $M=H_{1}\cup H_{2}$ be the Heegaard splitting of $M$, induced
by $\mathcal{H}$. Then for each $i=1,2$, there exists a 1-subcomplex
$C_{i}\subset\text{int}(H_{i})$ so that $H_{1}-C_{1}\simeq\Sigma\times[0,\infty)$
and $H_{2}-C_{2}\simeq\Sigma\times(-\infty,0]$. Since $C_{1}$ and
$C_{2}$ are both 1-dimensional and $M$ is a 3-manifold, we can isotope
$(L,p)$ so that $p\in\Sigma$, $L$ does not intersect $C_{1}\cup C_{2}$
and $L$ intersect transversely with $\Sigma$. After further isotoping
$L$, we may assume that every component $c^{i}$ of $L\cap H_{i}$
admits a disk $D_{c^{i}}\subset H_{i}$ so that $c^{i}\subset\partial D_{c^{i}}\subset c^{i}\cup\Sigma$,
and for any two distinct components $c_{1}^{i},c_{2}^{i}$ of $L\cap H_{i}$,
we have $D_{c_{1}^{i}}\cap D_{c_{2}^{i}}=\emptyset$. We can also
assume that the disks $D_{c^{i}}$ does not intersect the family of
compressing disks in $H_{i}$ , determined by the alpha- and beta-curves
of $\mathcal{H}$, possibly after applying another isotopy to $L$,
while leaving the basepoint $p$ fixed. For each component $c^{i}\in\pi_{0}(L\cap H_{i})$,
write $\partial D_{c^{i}}=c^{i}\cup p(c^{i})$ where $u^{i}\subset\Sigma$,
i.e. $p(c^{i})$ is a simple arc on $\Sigma$ which is a projection
of $c^{i}$. By assumption, the arcs $p(c^{i})$ does not intersect
the alpha- or beta-curves, and any two distinct arcs $p(c_{1}^{i})$
and $p(c_{2}^{i})$ do not intersect. 

Now, after isotoping the arcs $p(c^{i})$, we can assume that for
any two curves $p(c^{i})$ and $p(c^{j})$ intersect transversely
if $i\ne j$, i.e. $\{i,j\}=\{1,2\}$. Consider the following sets:
\begin{align*}
A & =\{p(c)\,\vert\,c\in\pi_{0}(L\cap H_{1})\},\\
B & =\{p(c)\,\vert\,c\in\pi_{0}(L\cap H_{2})\},\\
F & =\bigcup_{u\in A\cup B}\partial u.
\end{align*}
 Then $p\in F$, and the based bridge diagram $(A,B,F,p)$ represents
the given based link $(L,p)$ in $M$.
\end{proof}
\begin{defn}
An extended bridge diagram is a pair $\mathcal{E}=(\mathcal{H},\mathcal{P})$,
where $\mathcal{H}$ is a Heegaard diagram and $\mathcal{P}$ is a
based bridge diagram on $\mathcal{H}$. We write $\mathcal{H}$ as
$\mathcal{H}(\mathcal{E})$ and $\mathcal{P}$ as $\mathcal{P}(\mathcal{E})$.
If $\mathcal{H}=(\Sigma,\boldsymbol{\alpha},\boldsymbol{\beta})$
and $\mathcal{P}=(A,B,F,p)$, the pointed Heegaard diagram $(\Sigma,\boldsymbol{\alpha},\boldsymbol{\beta},p)$
will be denoted as $\mathcal{H}_{st}(\mathcal{E})$. Also, the 3-manifold
represented by $\mathcal{H}$ is denoted as $M(\mathcal{E})$, and
the based link in $M(\mathcal{E})$ represented by $\mathcal{P}$
is denoted as $L(\mathcal{E})$.
\end{defn}

Given an extended bridge diagram $\mathcal{E}=(\mathcal{H},\mathcal{P})$,
we have an associated 3-tuple $(M,L,p)$, where $M$ is the 3-manifold
represented by $\mathcal{H}$, and $(L,p)$ is a based link in $M$,
represented by $\mathcal{P}$. Obviously, given a 3-maniold together
with a based link inside it, we have lots of extended bridge diagrams
which represents it. In particular, we have a set of operations on
extended bridge diagrams which leave the associated 3-manifold and
based link fixed, which we will call as extended Heegaard moves. Also,
we will call isotopies and handleslides involving $\alpha$($\beta$)-curves
as $\alpha$($\beta$)-equivalences, and those involving A(B)-arcs
and A(B)-equivalences. Finally, we will call $\alpha$,$\beta$-equivalences
and A,B-equivalences as basic moves.

\subsection*{Isotopies}

Given an extended bridge diagram, we can isotope its A-arcs, B-arcs,
alpha-curves and beta-curves.

\subsection*{Handleslides of type I}

Given an extended bridge diagram, we can replace an alpha(beta)-curve
$\alpha$ with another simple closed curve $\alpha^{\prime}$ through
an ordinary handleslide of knot Heegaard diagrams. Here, the handleslide
region must not intersect any of the A/B-arcs.

\subsection*{Handleslides of type II}

Given an extended bridge diagram, we can replace an alpha(beta)-curve
$\alpha$ with another simple closed curve $\alpha^{\prime}$ if the
following conditions are satisfied.
\begin{itemize}
\item The curve $\alpha^{\prime}$ does not intersect with any of the A(B)-arcs
and the alpha(beta)-curves.
\item There exists an A-arc $A$ so that $\alpha,\alpha^{\prime}$ cobound
a cylinder whose interior contains $A$ and does not intersect with
any of the A-arcs, B-arcs, alpha-curves, and the beta-curves, except
$A$.
\end{itemize}

\subsection*{Handleslides of type III}

Given an extended bridge diagram, we can replace an A(B)-arc $a$
with another simple arc $a^{\prime}$ if the following conditions
are satisfied.
\begin{itemize}
\item $\partial a=\partial a^{\prime}$.
\item The interior of $a^{\prime}$ does not intersect with any of the A(B)-arcs
and the alpha(beta)-curves.
\item There exists an alpha(beta)-curve $\alpha$ so that $a,a^{\prime},\alpha$
bound a cylinder $C\subset\Sigma$, whose interior does not intersect
with any of the A-arcs, B-arcs, alpha-curves, and the beta-curves.
\end{itemize}

\subsection*{Handleslides of type IV}

Given an extended bridge diagram, we can replace an A(B)-arc $a$
with another simple arc $a^{\prime}$ if the following conditions
are satisfied.
\begin{itemize}
\item The interior of $a^{\prime}$ does not intersect with any of the A(B)-arcs
and the alpha(beta)-curves.
\item There exists an A(B)-arc $a_{0}$ such that $a,a^{\prime}$ bound
a disk $D\subset\Sigma$, whose interior contains $a_{0}$ and does
not intersect with any of the A-arcs, B-arcs, alpha-curves, and the
beta-curves, except for $a_{0}$.
\end{itemize}

\subsection*{(De)stabilizations of type I}

Given an extended bridge diagram $((\Sigma,\boldsymbol{\alpha},\boldsymbol{\beta}),(F,A,B,p))$,
we can stabilize its Heegaard diagram $(\Sigma,\boldsymbol{\alpha},\boldsymbol{\beta})$
at a point $q\in\Sigma$ such that $q\notin c$ for any $c\in A\cup B$.
The based bridge diagram $(F,A,B,p)$ remains the same.

\subsection*{(De)stabilizations of type II}

Given an extended bridge diagram $((\Sigma,\boldsymbol{\alpha},\boldsymbol{\beta}),(F,A,B,p))$,
where $|F|>2$, choose an A-arc $a$ such that $p\notin\partial a$,
and pick one of its endpoints, $z\in\partial a$. Choose two distinct
points $x,y$ lying in the interior of $z$, so that the following
conditions hold.
\begin{itemize}
\item $a-\{x,y\}$ has three components $a_{1},b_{1},a^{\prime}$, which
are simple arcs on $\Sigma$.
\item $\partial a_{1}=\{z,x\}$, $\partial b_{1}=\{x,y\}$, and $\partial a^{\prime}=\{y\}\cup(\partial a-\{z\})$.
\item $a_{1}$ and $b_{1}$ do not intersect with any of the B-arcs and
beta-curves.
\end{itemize}
Then $((\Sigma,\boldsymbol{\alpha},\boldsymbol{\beta}),(F\cup\{x,y\},(A-\{a\})\cup\{a_{1},a^{\prime}\}),B\cup\{b_{1}\},p))$
is again an extended bridge diagram.

\subsection*{Diffeomorphism}

Given an extended bridge diagram $((\Sigma,\boldsymbol{\alpha},\boldsymbol{\beta}),(F,A,B,p))$
and a diffeomorphism $\phi\in\text{Diff}^{+}(\Sigma)$, we can apply
$\phi$ on everything to get another extended bridge diagram.
\begin{prop}
\label{basicmoves1}Let $\mathcal{H}=(\Sigma,\boldsymbol{\alpha},\boldsymbol{\beta})$
be a Heegaard diagram which represent a 3-manifold $M$. Any two based
bridge diagrams on $\mathcal{H}$, which represent isotopic based
links in $M$, are related by isotopies, handleslides of type II,
III, IV, and (de)stabilizations of type II.
\end{prop}

\begin{proof}
Choose a pair of a self-indexing Morse function $f:M\rightarrow[0,3]$
and a Riemannian metric $g$ on $M$, which induces the Heegaard diagram
$\mathcal{H}$ of $M$. In the space $\mathcal{S}$ of based link
$M$ such that its basepoint lies on $\Sigma$ and it is transverse
to $\Sigma$ at the basepoint, the subspace $\mathcal{S}_{0}$ of
based links $(L,p)$ which satisfy the conditions below is open and
dense. Given any link in $\mathcal{S}_{0}$, its projection along
the gradient flow of $f$ gives a based bridge diagram of $(L,p)$
on $\mathcal{H}$, up to stabilizations of type II.
\begin{itemize}
\item The intersection $L\cap\Sigma$ is transverse.
\item The gradient vector field $\nabla_{g}f$ is nonvanishing on $L$ and
transverse to $L$.
\item For any flowline $c$ of $\nabla_{g}f$ whose endpoints lie on $L$,
the intersection $c\cap\Sigma$ is transverse.
\item For each bi-infinite flowline $\gamma$ of $\nabla_{g}f$, we have
$|\gamma\cap L|\le2$, and if the equality holds, we have $\gamma\cap\Sigma\cap L=\emptyset$.
\item The intersections of $L$ with the unstable manifolds of critical
points of index $2$ and the stable manifolds of critical points of
index $1$ are transverse.
\end{itemize}
We call the set $\mathcal{S}_{0}$ as the set of points of codimension
$0$; to prove the proposition, it suffices to classify the codimension
$1$ singularities inside $\mathcal{S}$ and show that they correspond
to (compositions of) handleslides of type II, III, IV, and stabilizations
of type II. It is easy to see that the codimension $1$ singularities
in $\mathcal{S}$ are given as follows.
\begin{enumerate}
\item The link $L$ is tangent to $\Sigma$ at a point $z\in\Sigma$, such
that $z\ne p$ and the order of tangency is $1$.
\item The link $L$ intersects transversely with either the stable manifold
of a critical point of index $2$ or the unstable manifold of a critical
point of index $1$.
\item There exists a flowline $\gamma$ of $\nabla_{g}f$ which is tangent
to $L$ at a point, such that the order of tangency is $1$.
\item There exists a bi-infinite flowline $\gamma$ of $\nabla_{g}f$ such
that $|\gamma\cap L|=2$ and $|\gamma\cap\Sigma\cap L|=1$.
\item The link $L$ is tangent to either the unstable manifold of a critical
point of index $2$ or the stable manifold of a critical point of
index $1$, such that the order of tangency is $1$.
\item There exists three distinct points $x,y,z\in\Sigma$, different from
the basepoint $p$, and a bi-infinite flowline $\gamma$ of $\nabla_{g}f$
such that $x,y,z\in\gamma$.
\end{enumerate}
The perturbations of the above singularities can be translated as
the following compositions of extended Heegaard moves.
\begin{enumerate}
\item A single stabilization of type II.
\item A single handleslide of type III.
\item The composition of two stabilization of type II and a handleslide
of type II.
\item The composition of a stabilization of type II and a handleslide of
type IV.
\item An isotopy.
\item The composition of a stablilization of type II, a handleslide of type
IV, and an isotopy.
\end{enumerate}
Therefore we see that any two based bridge diagrams on $\mathcal{H}$
which represent isotopic based links in $M$ are related by isotopies,
handleslides of type II, III, IV, and (de)stabilizations of type II.
\end{proof}
\begin{defn}
A based bridge diagram $\mathcal{P}$ on a Heegaard diagram $\mathcal{H}=(\Sigma,\boldsymbol{\alpha},\boldsymbol{\beta})$
is simple if all A-arcs and B-arcs of $\mathcal{P}$ lie in the same
connected component of $\Sigma-\left(\bigcup_{\gamma\in\boldsymbol{\alpha}\cup\boldsymbol{\beta}}\gamma\right)$.
\end{defn}

\begin{thm}
\label{basicmoves2}If two extended bridge diagrams which represent
the same 3-manifold and isotopic based links, which are contained
in a ball, they are related by extended Heegaard moves.
\end{thm}

\begin{proof}
Since the given link is assumed to be contained in a ball, for any
Heegaard diagram $\mathcal{H}$, there exists a simple based bridge
diagram $\mathcal{P}_{0}$ which also represents the given link. Now,
given any based bridge diagram $\mathcal{P}$ on a Heegaard diagram
$\mathcal{H}$, we know from Proposition \ref{basicmoves1} that we
can apply isotopies, handleslides of type II, III, IV, and (de)stabilizations
of type II to $\mathcal{P}$ to reach $\mathcal{P}_{0}$. But then,
we can regard the A-arcs and B-arcs together as a ``big basepoint''
and apply isotopies, handleslides, stabilizations and diffeomorphisms
to the Heegaard diagram $\mathcal{H}$. The handleslides and stabilizations
applied to $\mathcal{H}$ corresponds to handleslides of type I and
stabilizations of type I applied to the extended bridge diagram $(\mathcal{P},\mathcal{H})$.
Since any two Heegaard diagrams representing the same 3-manifold are
related by isotopies, handleslides, stabilizations, and diffeomorphisms,
the proof is complete.
\end{proof}
\begin{rem*}
By considering perturbations of Morse-Smale pairs on $M$ together
with perturbations of the given link $L$, and classifying all possible
codimension $1$ singularities, we can remove the the assumption that
our base link is contained in a ball, in Theorem \ref{basicmoves2}.
However, this observation is not necessary, as we will only consider
knots and links in $S^{3}$ throughout this paper.
\end{rem*}
Now suppose that an extended bridge diagram $\mathcal{E}=(\mathcal{H},\mathcal{P})$,
$\mathcal{H}=(\Sigma,\boldsymbol{\alpha},\boldsymbol{\beta})$, $\mathcal{P}=(F,A,B,z)$
is given, where the Heegaard diagram $\mathcal{H}$ represents a 3-manifold
$M$ and the based bridge diagram $\mathcal{P}$ represents the isotopy
class of a based link $(L,z)$ in $M$. Then we construct a $4$-tuple
$\mathcal{H}_{d}(\mathcal{E})=(\tilde{\Sigma},\tilde{\boldsymbol{\alpha}},\tilde{\boldsymbol{\beta}},z)$,
which is defined as follows.
\begin{itemize}
\item $\tilde{\Sigma}$ is the branched double cover of $\Sigma$ along
$F$.
\item $\tilde{\boldsymbol{\alpha}}=\left(\bigcup_{\alpha\in\boldsymbol{\alpha}}\{\text{connected components of }\alpha\}\right)\cup\{p^{-1}(a)\,\vert\,a\in A,\,z\notin\partial a\}$,
where $p\,:\,\tilde{\Sigma}\rightarrow\Sigma$ is the branched covering
map, and $\tilde{\boldsymbol{\beta}}$ is defined similarly.
\item $z$ is the basepoint of the based link $(L,z)$.
\end{itemize}
The $4$-tuple $\mathcal{H}_{d}(\mathcal{E})$, by construction, is
a Heegaard diagram for the based 3-manifold $(\Sigma_{L}(M),z)$,
which is the branched double cover of the based 3-manifold $(M,z)$
along the link $L$. The covering transformation of the branched cover
$\Sigma_{L}(M)\rightarrow M$ induces an orientation-preserving $\mathbb{Z}_{2}$-action
on $\mathcal{H}_{d}(\mathcal{E})$. We will say that $\mathcal{H}_{d}(\mathcal{E})$
is the branched double cover of $\mathcal{H}$ along $\mathcal{P}$.
\begin{prop}
\label{weakadm}For any extended bridge diagram $\mathcal{E}$, the
pointed Heegaard diagarm $\mathcal{H}_{d}(\mathcal{E})$ is weakly
admissible if $\mathcal{H}_{pt}(\mathcal{E})$ is weakly admissible.
\end{prop}

\begin{proof}
We will continue using the notations which we have used above. Consider
the branched covering map $p_{\Sigma}:\tilde{\Sigma}\rightarrow\Sigma$.
Then, for any connected component $R$ of $\Sigma-\left(\bigcup_{c\in\boldsymbol{\alpha}\cup\boldsymbol{\beta}}c\right)$,
which does not contain the basepoint $z$, we define its pullback
$p_{\Sigma}^{\ast}(R)$ as follows.
\begin{itemize}
\item If $p_{\Sigma}^{-1}(R)$ is connected, it is a connected component
of $\tilde{\Sigma}-\left(\bigcup_{\tilde{c}\in\tilde{\boldsymbol{\alpha}}\cup\tilde{\boldsymbol{\beta}}}\tilde{c}\right)$,
so we define $p_{\Sigma}^{\ast}(R)$ as $p_{\Sigma}^{-1}(R)$.
\item If $p_{\Sigma}^{-1}(R)$ is disconnected, it consists of a $\mathbb{Z}_{2}$-orbit
of some connected component of $\tilde{\Sigma}-\left(\bigcup_{\tilde{c}\in\tilde{\boldsymbol{\alpha}}\cup\tilde{\boldsymbol{\beta}}}\tilde{c}\right)$,
where $\mathbb{Z}_{2}$ acts as covering transformations. Denote that
orbit as $\{T,\sigma T\}$, where $\mathbb{Z}_{2}=\left\langle \sigma\right\rangle $.
Then we define $p_{\Sigma}^{\ast}(R)$ as $T+\sigma T$.
\end{itemize}
This definition can be extended linearly to give a group homomorphism
\[
p_{\Sigma}^{\ast}:\mathcal{D}(\mathcal{H}_{pt}(\mathcal{E}))\rightarrow\mathcal{D}(\mathcal{H}_{d}(\mathcal{E})),
\]
 where $\mathcal{D}(\mathcal{H})$ for a point Heegaard diagram $\mathcal{H}$
is defined to be the free abelian group of domains in $\mathcal{H}$
which do not intersect the basepoint. The map $p_{\Sigma}^{\ast}$
clearly preserves periodicity.

Suppose that $\mathcal{H}_{pt}(\mathcal{E})$ is weakly admissible
and there exists a positive periodic domain $D\in\mathcal{D}(\mathcal{H}_{d}(\mathcal{E}))$.
Then $D+\sigma D$ is also a positive periodic domain in $\mathcal{H}_{d}(\mathcal{E})$.
Using the proof of Lemma 4.2 in \cite{Kang}, we see that there exists
a positive periodic domain $D_{0}\in\mathcal{D}(\mathcal{H}_{pt}(\mathcal{E}))$
such that $p_{\Sigma}^{\ast}D_{0}=D+\sigma D$. Since $\mathcal{H}_{pt}(\mathcal{E})$
is assumed to be weakly admissible, we must have $D_{0}=0$ and thus
$D+\sigma D=0$. Since both $D$ and $\sigma D$ are positive, this
implies $D=0$, a contradiction. Therefore $\mathcal{H}_{d}(\mathcal{E})$
must be weakly admissible.
\end{proof}
We will now proceed to weak admissibilities of Heegaard triple diagrams
and quadruple diagrams, which are perturbations of branched double
covers of extended bridge diagrams. More precisely, the diagrams we
will deal with are defined as follows.
\begin{defn}
A 5-tuple $(\tilde{\Sigma},\tilde{\boldsymbol{\alpha}},\tilde{\boldsymbol{\beta}},\tilde{\boldsymbol{\gamma}},z)$
is called an involutive Heegaard 5-tuple if the conditions below are
satisfied. 
\begin{itemize}
\item $\tilde{\Sigma}$ is a branched double cover of a surface $\Sigma$
along a branching locus $F$, such that $z\in F$.
\item The 4-tuples $\tilde{\mathcal{H}}_{\alpha\beta}=(\tilde{\Sigma},\tilde{\boldsymbol{\alpha}},\tilde{\boldsymbol{\beta}},z)$,
$\tilde{\mathcal{H}}_{\beta\gamma}=(\tilde{\Sigma},\tilde{\boldsymbol{\beta}},\tilde{\boldsymbol{\gamma}},z)$,
$\tilde{\mathcal{H}}_{\alpha\gamma}=(\tilde{\Sigma},\tilde{\boldsymbol{\alpha}},\tilde{\boldsymbol{\gamma}},z)$
are pointed Heegaard diagrams.
\item There exist families of simple closed curves $\boldsymbol{\alpha},\boldsymbol{\beta},\boldsymbol{\gamma}$
and families of simple arcs $A,B,C$ on $\Sigma$ such that $\mathcal{T}_{0}=(\Sigma,\boldsymbol{\alpha},\boldsymbol{\beta},\boldsymbol{\gamma},z)$
is a Heegaard triple diagram, $\mathcal{P}_{AB}=(F,A,B,z)$, $\mathcal{P}_{BC}=(F,B,C,z)$,
and $\mathcal{P}_{AC}=(F,A,C,z)$ are based bridge diagrams on the
Heegaard diagrams $\mathcal{H}_{\alpha\beta}=(\Sigma,\boldsymbol{\alpha},\boldsymbol{\beta})$,
$\mathcal{H}_{\beta\gamma}=(\Sigma,\boldsymbol{\beta},\boldsymbol{\gamma})$,
and $\mathcal{H}_{\alpha\gamma}=(\Sigma,\boldsymbol{\alpha},\boldsymbol{\gamma})$,
respectively, and the branched double covers of $\mathcal{H}_{\alpha\beta},\mathcal{H}_{\beta\gamma},\mathcal{H}_{\alpha\gamma}$
along $\mathcal{P}_{AB},\mathcal{P}_{BC},\mathcal{P}_{CA}$ are $\tilde{\mathcal{H}}_{\alpha\beta},\tilde{\mathcal{H}}_{\beta\gamma},\tilde{\mathcal{H}}_{\alpha\gamma}$,
respectively.
\end{itemize}
A pointed Heegaard triple diagram $\mathcal{T}$ is nearly involutive
if it is given by a small perturbation of alpha-, beta-, and gamma-curves
of some involutive Heegaard 5-tuple $(\tilde{\Sigma},\tilde{\boldsymbol{\alpha}},\tilde{\boldsymbol{\beta}},\tilde{\boldsymbol{\gamma}},z)$.
We say that the pointed Heegaard triple diagram $\mathcal{T}_{0}$
is the base of the nearly involutive triple diagram $\mathcal{T}$.
\end{defn}

\begin{prop}
\label{weakadm2}A nearly involutive Heegaard triple diagram is weakly
admissible if its base is weakly admissible.
\end{prop}

\begin{proof}
The proof is the same as in Proposition \ref{weakadm}, except that
we are using triple diagrams instead of ordinary diagrams. Using the
proof of Lemma 4.3 in \cite{Kang}, we see that the argument for ordinary
diagrams can also be used for triple diagrams.
\end{proof}
\begin{defn}
A pointed 6-tuple $(\tilde{\Sigma},\tilde{\boldsymbol{\alpha}},\tilde{\boldsymbol{\beta}},\tilde{\boldsymbol{\gamma}},\tilde{\boldsymbol{\delta}},z)$
is an involutive Heegaard 6-tuple if any of the four 5-tuples given
by excluding one out of four curve bases $\tilde{\boldsymbol{\alpha}},\tilde{\boldsymbol{\beta}},\tilde{\boldsymbol{\gamma}},\tilde{\boldsymbol{\delta}}$
are involutive. A pointed Heegaard quadruple diagram $\mathcal{Q}$
is nearly admissible if it is given by a small perturbation of alpha-,
beta-, gamma-, and delta-curves of some involutive Heegaard 6-tuple
$(\tilde{\Sigma},\tilde{\boldsymbol{\alpha}},\tilde{\boldsymbol{\beta}},\tilde{\boldsymbol{\gamma}},\tilde{\boldsymbol{\delta}},z)$.
If the bases of the triple diagrams given by exluding one out of four
curve bases $\tilde{\boldsymbol{\alpha}},\tilde{\boldsymbol{\beta}},\tilde{\boldsymbol{\gamma}},\tilde{\boldsymbol{\delta}}$
are given by a surface $\Sigma$ and curve bases $\boldsymbol{\alpha},\boldsymbol{\beta},\boldsymbol{\gamma},\boldsymbol{\delta}$
on $\Sigma$, then we say that the pointed Heegaard quadruple diagram
$\mathcal{Q}_{0}=(\Sigma,\boldsymbol{\alpha},\boldsymbol{\beta},\boldsymbol{\gamma},\boldsymbol{\delta},z)$
is the base of the nearly involutive quadruple diagram $\mathcal{Q}$.
\end{defn}

\begin{prop}
\label{weakadm3}A nearly involutive Heegaard quadruple diagram is
weakly admissible if its base is weakly admissible.
\end{prop}

\begin{proof}
The proof is the same as in Proposition \ref{weakadm2}, except that
we are now using the proof of Lemma 4.4, instead of the proof of Lemma
4.3, of \cite{Kang}.
\end{proof}
The propositions \ref{weakadm}, \ref{weakadm2}, and \ref{weakadm3}
tell us that, when we deal with nearly involutive diagrams, we do
not have to care about their weak admissibility, as long as their
base are weakly admissible. Hence, for the sake of simplicity, we
will call a extended bridge diagram $\mathcal{E}$ weakly admissible
if the pointed Heegard diagram $\mathcal{H}_{pt}(\mathcal{E})$ is
weakly admissible. Note that this implies weak admissibility of $\mathcal{H}_{d}(\mathcal{E})$.

Now, given an extended bridge diagram 
\[
\mathcal{E}=((\Sigma,\boldsymbol{\alpha},\boldsymbol{\beta}),(F,A,B,z)),
\]
 such that $\mathcal{H}_{pt}(\mathcal{E})$ is weakly admissible,
we can apply the construction of \cite{eqv-Floer} to the induced
symplectic $\mathbb{Z}_{2}$-action on the triple $(\text{Sym}^{\tilde{g}}(\tilde{\Sigma}-\{z\}),\mathbb{T}_{\tilde{\boldsymbol{\alpha}}},\mathbb{T}_{\tilde{\boldsymbol{\beta}}})$,
where $\mathcal{H}_{d}(\mathcal{E})=(\tilde{\Sigma},\tilde{\boldsymbol{\alpha}},\tilde{\boldsymbol{\beta}},z)$
and $\tilde{g}$ is the genus of $\tilde{\Sigma}$. What we get is
the equivariant Floer cohomology 
\[
\widehat{HF}_{\mathbb{Z}_{2}}(\mathcal{E})=\widehat{HF}_{\mathbb{Z}_{2}}(\text{Sym}^{\tilde{g}}(\tilde{\Sigma}-\{z\}),\mathbb{T}_{\tilde{\boldsymbol{\alpha}}},\mathbb{T}_{\tilde{\boldsymbol{\beta}}}),
\]
 which is a $\mathbb{F}_{2}[\theta]$-module in a natural way.
\begin{lem}
\label{weakadm-moves}Given any two extended bridge diagram $\mathcal{E},\mathcal{E}^{\prime}$
representing the same bridge link $(L,p)$ inside the same 3-manifold
$M$, such that $\mathcal{H}_{pt}(\mathcal{E})$ and $\mathcal{H}_{pt}(\mathcal{E}^{\prime})$
are weakly admissible and $L$ is contained in a ball, there exists
a sequence of extended Heegaard moves which relates $\mathcal{E}$
and $\mathcal{E}^{\prime}$, such that for every extended bridge diagram
$\mathcal{E}_{0}$ appearing in an intermediate step, the pointed
Heegaard diagram $\mathcal{H}_{pt}(\mathcal{E}_{0})$ is weakly admissible.
\end{lem}

\begin{proof}
From Proposition \ref{weakadm} and the proof of Theorem \ref{basicmoves2},
we see that we do not have to consider based bridge diagrams of extended
bridge diagrams. So we only have to care about extended Heegaard moves
of pointed Heegaard diagrams. Since any two weakly admissible diagrams
representing the same 3-manifold are related by isotopies, handleslides,
stabilizations, and diffeomorphisms while preserving weak admissibility
by Proposition 2.2 of \cite{OSz-original}, we are done.
\end{proof}
Given an extended Heegaard move of extended bridge diagrams whose
source and target are both weakly admissible, we can associate it
to a $\mathbb{F}_{2}[\theta]$-module homomorphism between the corresponding
$\mathbb{Z}_{2}$-equivariant Floer cohomology, as defined in \cite{eqv-Floer}.
From Lemma \ref{weakadm-moves}, we know that any two weakly admissible
extended bridge diagrams $\mathcal{E},\mathcal{E}^{\prime}$ which
represent the same based link inside the same 3-manifold, we have
a map 
\[
\widehat{HF}_{\mathbb{Z}_{2}}(\mathcal{E}^{\prime})\rightarrow\widehat{HF}_{\mathbb{Z}_{2}}(\mathcal{E}),
\]
 which is defined as a composition of maps associated to extended
Heegaard moves. The arguments used in the section 6 of \cite{eqv-Floer}
can be extended directly to show that the maps associated to extended
Heegaard moves, except for stabilizations of type II, are isomorphisms.
\begin{rem*}
Here, we will assume that all stabilizations which we consider here
occur near the basepoint, to ensure that they clearly induce isomorphisms
of $\widehat{HF}_{\mathbb{Z}_{2}}$. In general, when we want to stabilize
in a region which is not close to the basepoint, we can also associate
it to an isomorphism by first performing it near the basepoint and
then moving it via a sequence of handleslides. Of course, such an
isomorphism is not unique; this problem will be resolved in the next
section.
\end{rem*}
We now claim that, in some special cases, we can prove that stabilizations
of type II induce isomorphisms. Note that, from now on, we will implicitly
assume all extended bridge diagrams to be \textbf{weakly admissible},
by which we will mean that its base is weakly admissible; this is
possible without loss of generality by Lemma \ref{weakadm-moves}.
\begin{lem}
\label{eqvtrans}Let $\mathcal{E}=((\Sigma,\boldsymbol{\alpha},\boldsymbol{\beta}),(F,A,B,z))$
be an extended bridge diagram, and let $F=F_{1}\cup F_{2}$ be the
unique partition of $F$ such that every $c\in A\cup B$ satisfies
$|\partial c\cap F_{1}|=|\partial c\cap F_{2}|=1$. Then applying
a stabilization of type II to $\mathcal{E}$ induces an isomorphism
of equivariant Floer cohomology.
\end{lem}

\begin{proof}
Recall that, in the paper \cite{eqv-Floer}, the proof that stabilizations
of bridge diagrams induce isomorphisms use equivariant transversality.
That proof can be directly extended to our case, so if the $\mathbb{Z}_{2}$-action
on $\mathcal{H}_{d}(\mathcal{E})$ achieves equivariant transversality,
then a stabilization of type II induces an isomorphisms.

For any domain $D$ of $\mathcal{H}_{d}(\mathcal{E})$ from a Floer
generator $\mathbf{x}$ to a generator $\mathbf{y}$, the Maslov index
formula in the paper \cite{cylindrical} reads:
\[
\mu(D)=n_{\mathbf{x}}(D)+n_{\mathbf{y}}(D)+e(D),
\]
 where $n_{\mathbf{x}},n_{\mathbf{y}}$ are point measures and $e$
is the Euler measure. Let $p:\tilde{\Sigma}\rightarrow\Sigma$ be
the branched double covering map with branching locus $F$. Suppose
that $D$ is $\mathbb{Z}_{2}$-invariant. Then $\mathbf{x}$ and $\mathbf{y}$
are also $\mathbb{Z}_{2}$-invariant, and thus we may assume without
loss of generality that $\mathbf{x}=\mathbf{x}^{\prime}\cup F_{1}$
and $\mathbf{y}=\mathbf{y}^{\prime}\cup F_{1}$, where $\mathbf{x}^{\prime},\mathbf{y}^{\prime}$
are Floer generators in $\mathcal{H}_{pt}(\mathcal{E})$. By the assumption
that every $c\in A\cup B$ satsfies $|\partial c\cap F_{1}|=|\partial c\cap F_{2}|$,
the Maslov index of the domain $p^{\ast}D$, as defined in the proof
of Proposition \ref{weakadm}, is given as follows.
\begin{align*}
\mu(p^{\ast}D) & =n_{\mathbf{x}}(p^{\ast}D)+n_{\mathbf{y}}(p^{\ast}D)+e(p^{\ast}D)\\
 & =(2n_{\mathbf{x}}(D)+\sum_{y\in F_{1}}n_{y}(D))+(2n_{\mathbf{y}}(D)+\sum_{y\in F_{1}}n_{y}(D))+(2e(D)-\sum_{y\in F}n_{y}(D))\\
 & =2\mu(D)+\left(\sum_{y\in F_{1}}n_{y}(D)-\sum_{y\in F_{2}}n_{y}(D)\right)\\
 & =2\mu(D).
\end{align*}
 Therefore, the hypothesis (EH-2) in \cite{eqv-Floer} is satisfied,
and thus $\mathcal{E}$ achieves equivariant transversality.
\end{proof}
Now we argue that, given any extended bridge diagram which represents
a knot in $S^{3}$, we can always adjust it to a position in which
stabilizations of type II induce isomorphisms.
\begin{defn}
\label{niceposition-def}An extended bridge diagram $\mathcal{E}=((\Sigma,\boldsymbol{\alpha},\boldsymbol{\beta}),(F,A,B,z))$
is said to be in a nice position if there exists an arc $c\in A\cup B$
such that the following conditions hold.
\begin{itemize}
\item $z\in\partial c$.
\item The interior of $c$ does not intersect with any of the B-arcs and
beta-curves.
\end{itemize}
\end{defn}

Given an extended bridge diagram $\mathcal{E}=((\Sigma,\boldsymbol{\alpha},\boldsymbol{\beta}),(F,A,B,z))$
which is in a nice position, choose an arc $c$ as in Definition \ref{niceposition-def},
and assume without loss of generality that $c$ is an A-arc. Let $p$
be an endpoint of $c$ such that $\partial c=\{p,z\}$. Choose two
distinct interior points $x,y$ of $c$ such that the following condition
is satisfied.
\begin{itemize}
\item $c-\{x,y\}$ has three connected components $c_{px},c_{xy},c_{yz}$,
each of which is a simple arc, satisfying $\partial c_{px}=\{p,x\}$,
$\partial c_{xy}=\{x,y\}$, and $\partial c_{yz}=\{y,z\}$.
\end{itemize}
Then, the pair $\mathcal{E}^{\prime}=((\Sigma,\boldsymbol{\alpha},\boldsymbol{\beta}),(F,(A-\{c\})\cup\{c_{xy},c_{yz}\},B\cup\{c_{px}\},z))$
is an extended bridge diagram, which represent the same based link
in a same 3-manifold as $\mathcal{E}$. 
\begin{defn}
We define the above operation as special stabilization, i.e. applying
a special stabilization to $\mathcal{E}$ gives $\mathcal{E}^{\prime}$.
\end{defn}

\begin{lem}
\label{spstab-lemma}Let $(K,p)$ be a based knot in $S^{3}$. Then
any two extended bridge diagrams representing $(K,p)$ are related
by isotopies, handleslides of type I, II, III, IV, stabilizations
of type I near the basepoint, diffeomorphisms, and special stabilizations.
\end{lem}

\begin{proof}
We only have to prove that we can use special stabilizations instead
of stabilizations of type II. Given any extended bridge diagram $\mathcal{E}$
representing $(K,p)$ in $S^{3}$, we can apply handleslide of type
II, III, and IV to place it in a nice position. Since $K$ is a knot
and thus has only one component, applying a stabilization of type
II at any point has the same effect as applying a special stabilization
and then moving the newly created pair of arcs to that point via isotopies
and handleslide of type III. Therefore a stabilization of type II
has the same effect as a composition of handleslides of type III,
IV, and special stabilizations. 
\end{proof}
\begin{lem}
\label{spstab-isom}Let $(K,p)$ be a based knot in $S^{3}$. Suppose
that an extended bridge diagram $\mathcal{E}$, which is in a nice
position, represents $(K,p)$ in $M$, and applying a special stabilization
to $\mathcal{E}$ gives another diagram $\mathcal{E}^{\prime}$. Then
there exists an associated isomorphism:
\[
\widehat{HF}_{\mathbb{Z}_{2}}(\mathcal{E}^{\prime})\xrightarrow{\sim}\widehat{HF}_{\mathbb{Z}_{2}}(\mathcal{E}).
\]
\end{lem}

\begin{proof}
The extended bridge diagrams $\mathcal{E}$ and $\mathcal{E}^{\prime}$,
near the basepoint $p$, are drawn in Figure \ref{fig01}. Since $\gamma_{zx}$
is the only B-arc/$\beta$-curve which intersects the A-arc $\gamma_{xy}$
of $\mathcal{H}_{d}(\mathcal{E}^{\prime})$, any Floer generator $\mathbf{x}^{\prime}$
of $\mathcal{H}_{d}(\mathcal{E}^{\prime})$ can be written as 
\[
\mathbf{x}^{\prime}=\mathbf{x}\cup\{x\},
\]
 where $\mathbf{x}$ is a uniquely determined Floer generator of $\mathcal{H}_{d}(\mathcal{E})$.

Now assume that we are using almost complex structures of form $\text{Sym}^{\tilde{g}}(\mathfrak{j})$,
where $\mathfrak{j}$ is an almost complex structure on the Heegaard
surface $\tilde{\Sigma}$ of $\mathcal{H}_{d}(\mathcal{E})$ (thus
also of $\mathcal{H}_{d}(\mathcal{E}^{\prime})$) and $\tilde{g}$
is the genus of $\tilde{\Sigma}$; this is possible since such structures
are enough to acheive transversality for all homotopy classes of Whitney
disks. Since the regions of $\mathcal{H}_{d}(\mathcal{E}^{\prime})$
which contains $x$ on its boundary necessarily contains the basepoint
$z$, any holomorphic disks from $\mathbf{x}\cup\{x\}$ to $\mathbf{y}\cup\{x\}$,
where $\mathbf{x},\mathbf{y}$ are Floer generators of $\mathcal{H}_{d}(\mathcal{E})$,
are actually holomorphic disks from $\mathbf{x}$ to $\mathbf{y}$,
as the point $x$ cancels out. Hence we have a homeomorphism 
\[
\mathcal{M}(\mathbf{x},\mathbf{y})\simeq\mathcal{M}(\mathbf{x}\cup\{x\},\mathbf{y}\cup\{x\}).
\]
This implies that the map between the equivariant Floer complexes,
\begin{align*}
\widehat{CF}_{\mathbb{Z}_{2}}(\mathcal{E}) & \rightarrow\widehat{CF}_{\mathbb{Z}_{2}}(\mathcal{E}^{\prime}),\\
\mathbf{x} & \mapsto\mathbf{x}\cup\{x\}
\end{align*}
 is an isomorphism. Therefore the induced map $\widehat{HF}_{\mathbb{Z}_{2}}(\mathcal{E}^{\prime})\rightarrow\widehat{HF}_{\mathbb{Z}_{2}}(\mathcal{E})$
is also an isomorphism.
\end{proof}
\begin{figure}
\resizebox{.4\textwidth}{!}{\includegraphics{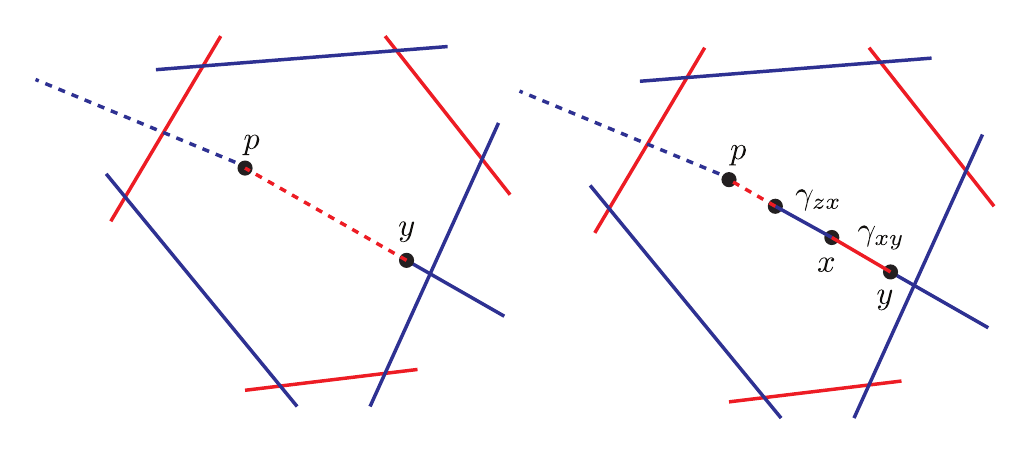}} \caption{\label{fig01}Attaching a small A-arc and a B-arc at the point $w$.}
\end{figure}

Now, given an extended bridge diagram $\mathcal{E}$ which represents
a knot in $S^{3}$, we have a following commutative diagram of extended
bridge diagrams, where the vertical arrows are stabilizations of type
II and horizontal arroLews are either isotopies, handleslides of type
I, II, III, IV, stabilizations of type I, or special stabilizations,
and $\mathcal{P}(\mathcal{E}_{n})$ (and also $\mathcal{P}(\mathcal{E}_{n}^{st})$)
are simple.
\[
\xymatrix{\mathcal{E}\ar[r]\ar[d]_{\text{stab}} & \mathcal{E}_{1}\ar[r]\ar[d]_{\text{stab}} & \cdots\ar[r] & \mathcal{E}_{n-1}\ar[r]\ar[d]^{\text{stab}} & \mathcal{E}_{n}\ar[d]^{\text{stab}}\\
\mathcal{E}^{st}\ar[r] & \mathcal{E}_{1}^{st}\ar[r] & \cdots\ar[r] & \mathcal{E}_{n-1}^{st}\ar[r] & \mathcal{E}_{n}^{st}
}
\]
 Translating this diagram into equivariant Floer cohomology and maps
between them gives the following diagram.
\[
\xymatrix{\widehat{HF}_{\mathbb{Z}_{2}}(\mathcal{E}_{n}^{st})\ar[r]^{\sim}\ar[d]^{\sim} & \widehat{HF}_{\mathbb{Z}_{2}}(\mathcal{E}_{n-1}^{st})\ar[r]^{\sim}\ar[d] & \cdots\ar[r]^{\sim} & \widehat{HF}_{\mathbb{Z}_{2}}(\mathcal{E}_{1}^{st})\ar[r]^{\sim}\ar[d] & \widehat{HF}_{\mathbb{Z}_{2}}(\mathcal{E}^{st})\ar[d]\\
\widehat{HF}_{\mathbb{Z}_{2}}(\mathcal{E}_{n})\ar[r]^{\sim} & \widehat{HF}_{\mathbb{Z}_{2}}(\mathcal{E}_{n-1})\ar[r]^{\sim} & \cdots\ar[r]^{\sim} & \widehat{HF}_{\mathbb{Z}_{2}}(\mathcal{E}_{1})\ar[r]^{\sim} & \widehat{HF}_{\mathbb{Z}_{2}}(\mathcal{E})
}
\]
 All horizontal arrows are isomorphisms. We also know that the leftmost
vertical arrow is also an isomorphism, since $\mathcal{P}(\mathcal{E})$
is simple. So, if this diagram is commutative, we can deduce that
the map $\widehat{HF}_{\mathbb{Z}_{2}}(\mathcal{E}^{st})\rightarrow\widehat{HF}_{\mathbb{Z}_{2}}(\mathcal{E}),$
which is the map induced by a stabilization of type II applied to
$\mathcal{E}$, is an isomorphism.
\begin{lem}
\label{associativity-1}Let $(\tilde{\Sigma},\tilde{\boldsymbol{\alpha}},\tilde{\boldsymbol{\beta}},\tilde{\boldsymbol{\gamma}},\tilde{\boldsymbol{\delta}},z)$
be an involutive Heegaard 6-tuple with base $(\Sigma,\boldsymbol{\alpha},\boldsymbol{\beta},\boldsymbol{\gamma},\boldsymbol{\delta},z)$,
and $\theta_{\beta,\gamma},\theta_{\gamma,\delta}$ be $\mathbb{Z}_{2}$-invariant
cycles in $\widehat{CF}(\tilde{\Sigma},\tilde{\boldsymbol{\beta}},\tilde{\boldsymbol{\gamma}},z),\widehat{CF}(\tilde{\Sigma},\tilde{\boldsymbol{\gamma}},\tilde{\boldsymbol{\delta}},z)$,
respectively. Suppose that, for any Floer generator $\mathbf{x}$
of $(\tilde{\Sigma},\tilde{\boldsymbol{\beta}},\tilde{\boldsymbol{\delta}},z)$,
every element of $\pi_{2}(\theta_{\beta,\gamma},\theta_{\gamma,\delta},\mathbf{x})$
has Maslov index at least $0$ and achieves transversality for generic
$\mathbb{Z}_{2}$-equivariant families of almost complex structures.
Then, for any $x_{\alpha,\beta}\in H_{\ast}(\widetilde{CF}_{\mathbb{Z}_{2}}(\tilde{\Sigma},\tilde{\boldsymbol{\alpha}},\tilde{\boldsymbol{\beta}})\otimes_{\mathbb{F}_{2}[\mathbb{Z}_{2}]}\mathbb{F}_{2})$,
we have 
\[
\hat{f}_{\alpha,\gamma,\delta}(\hat{f}_{\alpha,\beta,\gamma}(x_{\alpha,\beta}\otimes\theta_{\beta,\gamma})\otimes\theta_{\gamma,\delta})=\hat{f}_{\alpha,\beta,\delta}(x_{\alpha,\beta}\otimes f_{\beta,\gamma,\delta}(\theta_{\beta,\gamma}\otimes\theta_{\gamma,\delta})),
\]
 where $\hat{f}$ denotes the equivariant triangle maps, defined in
\cite{eqv-Floer}, and $f$ denotes the ordinary triangle map, with
repsect to their subindices.
\end{lem}

\begin{proof}
By assumption, the cycle $f_{\beta,\gamma,\delta}(\theta_{\beta,\gamma}\otimes\theta_{\gamma,\delta})\in\widehat{CF}_{\mathbb{Z}_{2}}(\tilde{\Sigma},\tilde{\boldsymbol{\beta}},\tilde{\boldsymbol{\delta}},z)$
is $\mathbb{Z}_{2}$-invariant, so that the term 
\[
\hat{f}_{\alpha,\beta,\delta}(x_{\alpha,\beta}\otimes f_{\beta,\gamma,\delta}(\theta_{\beta,\gamma}\otimes\theta_{\gamma,\delta}))
\]
 is well-defined. Using square maps from non-equivariant Floer theory,
we can construct the equivariant square map 
\[
\hat{s}_{\alpha,\beta,\gamma,\delta}\,:\,\widetilde{CF}_{\mathbb{Z}_{2}}(\tilde{\Sigma},\tilde{\boldsymbol{\alpha}},\tilde{\boldsymbol{\beta}},z)\rightarrow\widetilde{CF}_{\mathbb{Z}_{2}}(\tilde{\Sigma},\tilde{\boldsymbol{\alpha}},\tilde{\boldsymbol{\delta}},z),
\]
 by mimicing the consruction of equivariant triangle maps, as in the
proof of Lemma 3.25 in \cite{eqv-Floer}. If a sequence of holomorphic
squares in $(\tilde{\Sigma},\tilde{\boldsymbol{\alpha}},\tilde{\boldsymbol{\beta}},\tilde{\boldsymbol{\gamma}},\tilde{\boldsymbol{\delta}},z)$
having $\theta_{\beta,\gamma},\theta_{\gamma,\delta}$ as two of their
vertices degenerates to a concatenation of a holomorphic triangle
in $(\tilde{\Sigma},\tilde{\boldsymbol{\alpha}},\tilde{\boldsymbol{\beta}},\tilde{\boldsymbol{\delta}},z)$
and another holomorphc triangle in $(\tilde{\Sigma},\tilde{\boldsymbol{\beta}},\tilde{\boldsymbol{\gamma}},\tilde{\boldsymbol{\delta}},z)$,
the singular vertex should be the points in the cycle $f_{\beta,\gamma,\delta}(\theta_{\beta,\gamma}\otimes\theta_{\gamma,\delta})$
by assumption. Hence, if we choose a cycle representative $\mathbf{x}_{\alpha,\beta}$
of $x_{\alpha,\beta}$, the quantity
\[
\hat{f}_{\alpha,\gamma,\delta}(\hat{f}_{\alpha,\beta,\gamma}(\mathbf{x}_{\alpha,\beta}\otimes\theta_{\beta,\gamma})\otimes\theta_{\gamma,\delta})+\hat{f}_{\alpha,\beta,\delta}(\mathbf{x}_{\alpha,\beta}\otimes f_{\beta,\gamma,\delta}(\theta_{\beta,\gamma}\otimes\theta_{\gamma,\delta}))
\]
 should be the same as the quantity 
\[
d\hat{s}_{\alpha,\beta,\gamma,\delta}(\mathbf{x}_{\alpha,\beta}\otimes\theta_{\beta,\gamma}\otimes\theta_{\gamma,\delta})+\hat{s}_{\alpha,\beta,\gamma,\delta}(d\mathbf{x}_{\alpha,\beta}\otimes\theta_{\beta,\gamma}\otimes\theta_{\gamma,\delta}).
\]
 Therefore, by passing to the homology, we get the desired result.
\end{proof}
\begin{lem}
\label{comm-lemma-stabII}The maps induced by extended Heegaard moves,
and special stabilizations commute with the maps induced by stabilizations
of type II.
\end{lem}

\begin{proof}
The statement is obvious for diffeomorphisms and special stabilizations.
Also, the statement for isotopy maps is a direct consequence of Lemma
6.2 of \cite{Kang}. The remaining cases involve associativity of
equivariant triangle maps, and thus we would like to use Lemma \ref{associativity-1}.

Here, we will only give a proof for the commutation of stabilization
maps of type II with handleslides of type IV, since other cases can
be proven using the same argument. To give a proof for this case,
we consider the two triple-diagram in Figure \ref{fig02}. In (the
branched double cover of) the given triple-diagrams, there are clearly
no nonconstant triangles, and the constant triangles have Maslov index
$0$. Also, the induced (non-equivariant) triangle maps sends the
top classes to the unique $\mathbb{Z}_{2}$-invariant top class. Therefore,
by Lemma \ref{associativity-1}, the maps induced by extended Heegaard
moves and special stabilizations commute with the maps induced by
stabilizations of type II.
\end{proof}
\begin{figure}
\resizebox{.4\textwidth}{!}{\includegraphics{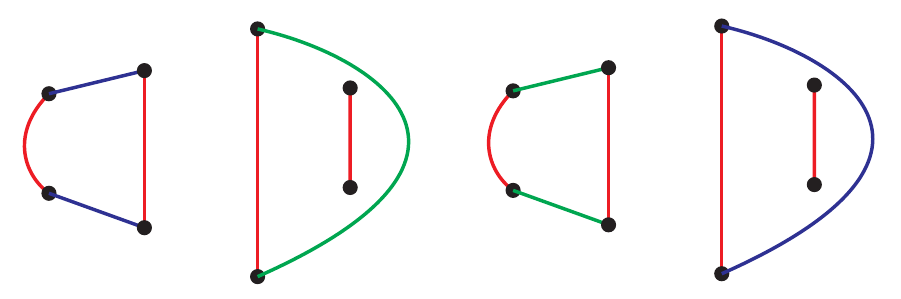}} \caption{\label{fig02}The diagram on the left is the triple-diagram for the
case when a stabilization map of type II is composed first and then
a handleslide map of type IV is composed, and the diagram on the right
is for the case when the maps are composed in the opposite order.}
\end{figure}
\begin{thm}
\label{thm:basicmove-isom}Given an extended bridge diagram $\mathcal{E}$
which represents a knot in $S^{3}$, let $\mathcal{E}^{\prime}$ be
the extended bridge diagram we get by applying an extended Heegaard
move to $\mathcal{E}$. Suppose that both $\mathcal{E}$ and $\mathcal{E}^{\prime}$
have at least one A-arcs and B-arcs. Then the associated map $\widehat{HF}_{\mathbb{Z}_{2}}(\mathcal{E}^{\prime})\rightarrow\widehat{HF}_{\mathbb{Z}_{2}}(\mathcal{E})$
is an isomorphism.
\end{thm}

\begin{proof}
We only have to consider the case when $\mathcal{E}^{\prime}$ is
obtained from $\mathcal{E}$ by a stabilization of type II. In this
case, recall that we had the following diagram. 
\[
\xymatrix{\widehat{HF}_{\mathbb{Z}_{2}}(\mathcal{E}_{n}^{st})\ar[r]^{\sim}\ar[d]^{\sim} & \widehat{HF}_{\mathbb{Z}_{2}}(\mathcal{E}_{n-1}^{st})\ar[r]^{\sim}\ar[d] & \cdots\ar[r]^{\sim} & \widehat{HF}_{\mathbb{Z}_{2}}(\mathcal{E}_{1}^{st})\ar[r]^{\sim}\ar[d] & \widehat{HF}_{\mathbb{Z}_{2}}(\mathcal{E}^{st})\ar[d]\\
\widehat{HF}_{\mathbb{Z}_{2}}(\mathcal{E}_{n})\ar[r]^{\sim} & \widehat{HF}_{\mathbb{Z}_{2}}(\mathcal{E}_{n-1})\ar[r]^{\sim} & \cdots\ar[r]^{\sim} & \widehat{HF}_{\mathbb{Z}_{2}}(\mathcal{E}_{1})\ar[r]^{\sim} & \widehat{HF}_{\mathbb{Z}_{2}}(\mathcal{E})
}
\]
 We know from Lemma \ref{comm-lemma-stabII} that this diagram is
commutative. Therefore all vertical arrows should be isomorphisms.
In particular, the map $\widehat{HF}_{\mathbb{Z}_{2}}(\mathcal{E}^{\prime})\rightarrow\widehat{HF}_{\mathbb{Z}_{2}}(\mathcal{E})$,
induced by a stabilization of type II, is an isomorphism.
\end{proof}
\begin{cor}
\label{cor:allgenus-isom}Let $\mathcal{E}$ be an extended bridge
diagrams representing a knot in $S^{3}$, which has at least two A-arcs.
Then $\widehat{HF}_{\mathbb{Z}_{2}}(\mathcal{E})\simeq\widehat{HF}_{\mathbb{Z}_{2}}(\Sigma(K))$.
\end{cor}

\begin{proof}
This follows directly from Theorem \ref{thm:basicmove-isom} and Proposition
\ref{basicmoves2}.
\end{proof}

\section{Equivariant Heegaard Floer cohomology for knots is a strong Heegaard
invariant}

Given a based knot $(K,z)$ in $S^{3}$, choose a knot Heegaard diagram
$H=(\Sigma,\boldsymbol{\alpha},\boldsymbol{\beta},z,w)$, so that
one of the two basepoints of $H$ (which is $z$ here) is the basepoint
of the given based knot. By distinguishing the two basepoints $z$
and $w$, we may say that $\mathcal{H}$ represents the based knot
$(K,z)$.
\begin{defn}
We say that a knot Heegaard diagram $H=(\Sigma,\boldsymbol{\alpha},\boldsymbol{\beta},z,w)$
represents a based knot $(K,p)$ in $S^{3}$ if $H$ repesents $K$
and $z=p$. Here, we call $w$ as the \textbf{second basepoint}.
\end{defn}

Given such a diagram $H$, we can construct an element $w\in H^{2}(\Sigma\backslash\{z,w\};\mathbb{F}_{2})$
by the following formula.
\begin{align*}
w([\alpha_{i}])=w([\beta_{j}])=0 & \text{ for each }\alpha_{i}\in\boldsymbol{\alpha}\text{ and }\beta\in\boldsymbol{\beta}\\
w([c])=1 & \text{ where }c\text{ is a small circle around }z
\end{align*}
 Since $(\Sigma,\boldsymbol{\alpha},\boldsymbol{\beta},z)$ represents
$S^{3}$, the homology classes $[\alpha_{i}]$ and $[\beta_{j}]$
span $H_{2}(\Sigma\backslash\{z,w\};\mathbb{F}_{2})$, so the above
formula defines $w$ uniquely. The element $w$ defines a map $\pi_{1}(\Sigma\backslash\{z-w\})\rightarrow\mathbb{Z}_{2}$,
thus induces a branched double covering $\tilde{\Sigma}\xrightarrow{f}\Sigma$,
branched along $\{z,w\}$. Denote the sets of inverse images of alpha-curves
and beta-curves by $\tilde{\boldsymbol{\alpha}}$ and $\tilde{\boldsymbol{\beta}}$.
Also, add a small A-arc $a$ and a B-arc $b$ at $w$, as drawn in
Figure \ref{fig03}. Then the Heegaard diagram $\tilde{H}=((\tilde{\Sigma},\tilde{\boldsymbol{\alpha}},\tilde{\boldsymbol{\beta}}),(\{z,w\},\{a\},\{b\},z))$
is a Heegaard diagram of $\Sigma(K)$ with a $\mathbb{Z}_{2}$-action.

Since $H$ can also be regarded as a 1-bridge knot diagram of $K$
with a hidden pair of bridges $a^{\prime}$ and $b^{\prime}$ connected
to $z$, the pair $\mathcal{E}_{H}=((\tilde{\Sigma},\tilde{\boldsymbol{\alpha}},\tilde{\boldsymbol{\beta}}),(\{z,w\},\{a,a^{\prime}\},\{b,b^{\prime}\},z))$
is an extended bridge diagram representing $K$. By the definition
of $\widehat{HF}_{\mathbb{Z}_{2}}(\mathcal{E}_{H})$, we have 
\[
\widehat{HF}_{\mathbb{Z}_{2}}(\mathcal{E}_{H})\simeq\widehat{HF}_{\mathbb{Z}_{2}}(\mathcal{H}_{d}(\tilde{H})).
\]
 Together with Corollary \ref{cor:allgenus-isom}, this gives a way
to compute the equivariant Heegaard Floer cohomology of a knot from
its knot Heegaard Floer cohomology.

\begin{figure}
\resizebox{.4\textwidth}{!}{\includegraphics{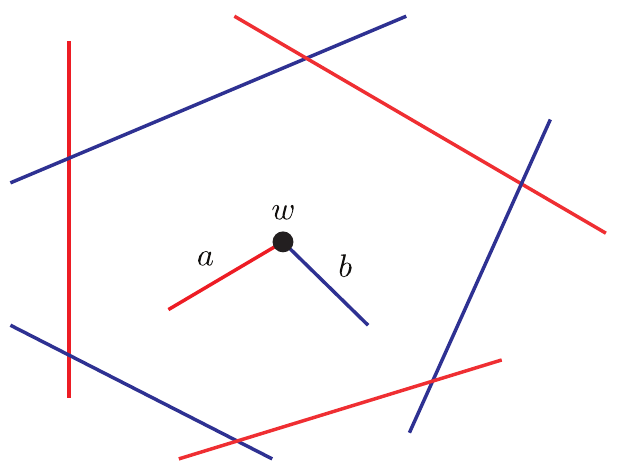}} \caption{\label{fig03}Attaching a small A-arc and a B-arc at the point $w$.}
\end{figure}
\begin{lem}
Let $H$ be a knot Heegaard diagram of a based knot $(K,z)$ in $S^{3}$.
Then $\widehat{HF}_{\mathbb{Z}_{2}}(\Sigma(K),z)\simeq\widehat{HF}_{\mathbb{Z}_{2}}(\tilde{H})$.
\end{lem}

\begin{proof}
By Corollary \ref{cor:allgenus-isom}, we have $\widehat{HF}_{\mathbb{Z}_{2}}(\Sigma(K),z)\simeq\widehat{HF}_{\mathbb{Z}_{2}}(\mathcal{E}_{H})\simeq\widehat{HF}_{\mathbb{Z}_{2}}(\tilde{H})$.
\end{proof}
Recall that any two knot Heegaard diagrams of a given based knot are
related by a sequence of Heegaard moves, namely, isotopies, handleslides,
(de)stabilizations, and diffeomorphisms. We claim that these basic
moves induce isomorphisms of $\widehat{HF}_{\mathbb{Z}_{2}}$, thereby
proving that the equivariant Heegaard Floer cohomology is a weak Heegaard
invariant.
\begin{thm}
\label{thm:basicmoves-knotHeegaard}Let $H,H^{\prime}$ be two knot
Heegaard diagrams of a based knot $(K,z)$, related by a Heegaard
move. Then there exists an induced isomorphism 
\[
\hat{f}_{\mathcal{H}\rightarrow\mathcal{H}^{\prime}}^{basic}\,:\,\widehat{HF}_{\mathbb{Z}_{2}}(\tilde{H})\rightarrow\widehat{HF}_{\mathbb{Z}_{2}}(\tilde{H}^{\prime}).
\]
\end{thm}

\begin{proof}
The isotopy case is obvious, and the handleslide case follows directly
from Proposition 3.24 of \cite{eqv-Floer}, which proves the invariance
of equivariant Floer cohomology under equivariant Hamiltonian isotopies
of Lagrangians. The stabilization case is just an application of the
standard neck-stretching argument, after assuming that it happens
near the basepoint.

It remains to show the case when the given Heegaard move is a diffeomorphism.
A diffeomorphism between (knot) Heegaard diagrams has two possible
lifts to diffeomorphisms between the branched double covers, and the
two lifts differ by the $\mathbb{Z}_{2}$-action. But the $\mathbb{Z}_{2}$-action
on $\tilde{H}$ induces the $\mathbb{F}_{2}[\mathbb{Z}_{2}]$-module
structure on the freed Floer complex, which is then dualized and tensored
with $\mathbb{F}_{2}$ with the trivial $\mathbb{Z}_{2}$-action when
calculating the cochain complex $\widehat{CF}_{\mathbb{Z}_{2}}^{\ast}(\tilde{H})$,
so that the $\mathbb{Z}_{2}$-action induced on $\widehat{HF}_{\mathbb{Z}_{2}}(\tilde{H})$
is trivial. Therefore the map between $\widehat{HF}_{\mathbb{Z}_{2}}(\tilde{H})$
induced by a diffeomorphism is well-defined.
\end{proof}
\begin{cor}
The $\mathbb{Z}_{2}$-equivariant Heegaard Floer cohomology of based
knots in $S^{3}$ is a weak Heegaard invariant in the sense of Juhasz;
see \cite{Juhasz-naturality} for the definition of Heegaard invariants.
\end{cor}

\begin{proof}
The statement of Theorem \ref{thm:basicmoves-knotHeegaard} is precisely
the definition of a weak Heegaard invariant.
\end{proof}
\begin{rem*}
Note that, due to the construction of stabilization maps of type I,
we do not know at this stage whether such maps are defined uniquely.
This problem will be resolved later.
\end{rem*}
Now we will show that the $\mathbb{Z}_{2}$-equivariant Heegaard Floer
cohomology is not only a weak invariant, but also a strong Heegaard
invariant. To recall the definition of strong Heegaard invariants,
we consider the graph $\mathcal{G}=\mathcal{G}(K,z)$, for a given
based knot $(K,z)$ in $S^{3}$, defined as follows.
\begin{align*}
V(\mathcal{G})= & \text{knot Heegaard isotopy diagrams representing }(K,z)\\
E(\mathcal{G})=E( & \mathcal{G}_{\alpha})\cup E(\mathcal{G}_{\beta})\cup E(\mathcal{G}_{\text{stab}})\cup E(\mathcal{G}_{\text{diff}})
\end{align*}
 Here, isotopy diagram means a diagram in which $\alpha$- and $\beta$-curves
are determined up to isotopy, and the graphs $\mathcal{G}_{\alpha},\mathcal{G}_{\beta},\mathcal{G}_{\text{stab}},\mathcal{G}_{\text{diff}}$
are graphs with the same vertex set as $\mathcal{G}$, defined in
the following way. $\mathcal{G}_{\alpha}$ is the graph such that
for each pair of elements $H_{1},H_{2}\in V(\mathcal{G})$, $\mathcal{G}_{\alpha}(H_{1},H_{2})$
consists of single element if $H_{1},H_{2}$ are $\alpha$-equivalent,
i.e. related by a sequence of $\alpha$-equivalences, and is empty
otherwise. $\mathcal{G}_{\beta}$ is the defined in the same way,
using $\beta$-equivalences. The subgraph $\mathcal{G}_{\text{stab}}$
is the graph of stabilizations, and $\mathcal{G}_{\text{diff}}$ is
the graph of diffeomorphisms. 

Also, there is a set $\mathcal{D}_{\mathcal{G}}$ of certain subgraphs
of $\mathcal{G}$, called distinguished rectangles, defined in Definition
2.30 of \cite{Juhasz-naturality}. Then we have the following definition.
\begin{defn}
\label{def:stronginv}(Definition 2.33 of \cite{Juhasz-naturality})
A weak Heegaard invariant $F$ (of knots in $S^{3}$) is a strong
Heegaard invariant if the following axioms hold.
\begin{itemize}
\item (Functoriality) The restriction of $F$ to $\mathcal{G}_{\alpha},\mathcal{G}_{\beta},\mathcal{G}_{\text{diff}}$
are functors, and for any stabilization $e$ and its reverse destabilization
$e^{\prime}$, we have $F(e^{\prime})=F(e)^{-1}$.
\item (Commutativity) For any distinguished rectangle $D\in\mathcal{D}_{\mathcal{G}}$,
the diagram $F(D)$ is commutative.
\item (Continuity) If $e$ is a loop-edge of $\mathcal{G}$ which is a diffeomorphism
isotopic to the identity, then $F(e)$ is the identity.
\item (Handleswap invariance) For every simple handleswap (see Figure 4
in \cite{Juhasz-naturality} for its definition) 
\[
\xymatrix{H_{1}\ar[rd]^{e}\\
H_{3}\ar[u]^{g} & H_{2}\ar[l]^{f}
}
\]
 in $\mathcal{G}$, we have $F(g)\circ F(f)\circ F(e)=\text{id}$.
\end{itemize}
\end{defn}

For the definition of distinguished rectangles and handleswaps, see
section 2.4 of \cite{Juhasz-naturality}.
\begin{lem}
\label{associativity-2}Let $(\tilde{\Sigma},\tilde{\boldsymbol{\alpha}},\tilde{\boldsymbol{\beta}},\tilde{\boldsymbol{\gamma}},\tilde{\boldsymbol{\delta}},z)$
be an involutive Heegaard 6-tuple with base $(\Sigma,\boldsymbol{\alpha},\boldsymbol{\beta},\boldsymbol{\gamma},\boldsymbol{\delta},z)$,
and $\theta_{\alpha,\beta},\theta_{\gamma,\delta}$ be $\mathbb{Z}_{2}$-invariant
cycles in $\widehat{CF}(\tilde{\Sigma},\tilde{\boldsymbol{\alpha}},\tilde{\boldsymbol{\beta}},z),\widehat{CF}(\tilde{\Sigma},\tilde{\boldsymbol{\gamma}},\tilde{\boldsymbol{\delta}},z)$,
respectively. Then, for any $x_{\beta,\gamma}\in H_{\ast}(\widetilde{CF}_{\mathbb{Z}_{2}}(\tilde{\Sigma},\tilde{\boldsymbol{\beta}},\tilde{\boldsymbol{\gamma}})\otimes_{\mathbb{F}_{2}[\mathbb{Z}_{2}]}\mathbb{F}_{2})$,
we have 
\[
\hat{f}_{\alpha,\gamma,\delta}(\hat{f}_{\alpha,\beta,\gamma}(\theta_{\alpha,\beta}\otimes x_{\beta,\gamma})\otimes\theta_{\gamma,\delta})=\hat{f}_{\alpha,\beta,\delta}(\theta_{\alpha,\beta}\otimes f_{\beta,\gamma,\delta}(x_{\beta,\gamma}\otimes\theta_{\gamma,\delta})),
\]
 where $\hat{f}$ denotes the equivariant triangle maps, defined in
\cite{eqv-Floer}, and $f$ denotes the ordinary triangle map, with
repsect to their subindices.
\end{lem}

\begin{proof}
The proof is basically the same as in Lemma \ref{associativity-1},
except that we do not need additional condition on triangles.
\end{proof}
\begin{lem}
\label{lem:alpha-beta-comm}Consider a following commutative diagram
$\mathcal{D}$ of extended bridge diagram of a based knot in $S^{3}$.
\[
\xymatrix{\mathcal{E}_{1}\ar[r]^{e}\ar[d]_{f} & \mathcal{E}_{2}\ar[d]^{f^{\prime}}\\
\mathcal{E}_{3}\ar[r]_{e^{\prime}} & \mathcal{E}_{4}
}
\]
 Suppose that $e,e^{\prime}$ are $\alpha$- or A-equivalences and
$f,f^{\prime}$ are $\beta$- or B-equivalences. Then $\widehat{HF}_{\mathbb{Z}_{2}}(\mathcal{D})$
commutes, i.e. 
\[
\widehat{HF}_{\mathbb{Z}_{2}}(e)\circ\widehat{HF}_{\mathbb{Z}_{2}}(f^{\prime})=\widehat{HF}_{\mathbb{Z}_{2}}(f)\circ\widehat{HF}_{\mathbb{Z}_{2}}(e^{\prime}).
\]
\end{lem}

\begin{proof}
This is a direct application of Lemma \ref{associativity-2}.
\end{proof}
\begin{lem}
\label{lem:alpha-functor}The restriction of the weak Heegaard invariant
$\widehat{HF}_{\mathbb{Z}_{2}}$ to $\mathcal{G}_{\alpha}$ is a functor.
\end{lem}

\begin{proof}
In this proof, we will use knot bridge diagrams of nontrivial genera
and their $\widehat{HF}_{\mathbb{Z}_{2}}$. The idea is to use the
naturality of $\widehat{HF}_{\mathbb{Z}_{2}}$ for bridge diagrams
on a sphere, which is already established by the author in \cite{Kang}.

Let $\{f_{i}\,:\,H_{i}\rightarrow H_{i+1}\,\vert\,i=1,\cdots,n\}$
be a composable sequence of $\alpha$-equivalences of knot Heegaard
diagrams of a based knot $(K,z)$ in $S^{3}$. Each $f_{i}$ is either
an isotopy of $\alpha$-curves or an $\alpha$-handleslide, which
can be translated to a 1-parameter family of Morse-Smale pairs on
$S^{3}$, as in Proposition 6.35 of \cite{Juhasz-naturality}. Composing
them gives a 1-parameter family $(f_{t},v_{t})_{t\in[0,1]}$ of simple
pairs, where $(f_{t},v_{t})$ is Morse-Smale except for finitely many
$t$. Also, there exists a Heegaard surface $\Sigma\subset S^{3}$
which intersects transversely with $K$, contains the basepoint $z$,
and is shared by every $(f_{t},v_{t})$, since no stabilizations or
destabilizations occur.

Given such a family and a Heegaard surface $\Sigma$, we now choose
an isotopy $\{K_{t}\}_{t\in[0,1]}$ of the given based knot $(K,z)$
while fixing the basepoint $z$, such that $K_{0}=K$, $K_{1}$ is
an isotopic copy of $K$ which is contained in a very small neighborhood
of $z$ and intersects transversely with $\Sigma$. We can further
assume that the given isotopy $\{K_{t}\}$ is generic, so that $K_{t}$
is transverse to $\Sigma$ for all but finitely many $t$. Since codimension-1
singularities of $\{K_{t}\}$ correspond to simple tangencies, we
know that, when passing through those ``finitely many'' $t$, a
pair of intersection points in $K_{t}\cap\Sigma$ is either created
or annihilated.

Now we multiply the two 1-parameter families mentioned above to get
a smooth 2-parameter family $\mathcal{F}$;
\[
\mathcal{F}=\{(f_{t},v_{t},K_{s})\}_{(t,s)\in[0,1]\times[0,1]}.
\]
 We can then perturb $\mathcal{F}$ in the space 
\[
\mathcal{S}=\{\text{simple pairs in }S^{3}\}\times\{\text{knots in }S^{3},\text{ isotopic to }K\text{ and containing }z\}
\]
to another generic 2-parameter family $\mathcal{F}^{\prime}$ close
to $\mathcal{F}$ in $\mathcal{S}$, which would have finitely many
codimension-2 singularities and no higher singularity. But since both
$(f_{t},v_{t})$ and $\{K_{s}\}$ were generic, by taking the perturbation
to be sufficiently small, we may assume that every codimension-2 singularity
arising in $\mathcal{F}^{\prime}$ is a combination of a codimension-1
singularity in the space of simple pairs in $S^{3}$ and a codimension-1
singularity in the space of based knots in $S^{3}$. Furthermore,
by closeness, no stabilization/destabilization would occur as a codimension-1
singularity. Therefore we only have to consider the following two
possible types of codimension-2 singularities in $\mathcal{S}$.
\begin{enumerate}
\item An $\alpha$-handleslide occurs and $K_{s}$ is (simple) tangent to
$\Sigma$; $K_{s}$ is transverse to all stable/unstable manifolds
of $v_{t}$.
\item An $\alpha$-handleslide occurs and $K_{s}$ intersects transversely
with the unstable manifold of a order 1 singularity of $v_{t}$ or
the stable manifold of an order 2 singularity of $v_{t}$.
\item An $\alpha$-handleslide occurs and $K_{s}$ is tangent to the stable
manifold of a order 1 singularity of $v_{t}$ or the unstable manifold
of an order 2 singularity of $v_{t}$.
\item An $\alpha$-handleslide occurs, $K_{s}$ is transverse to $\Sigma$
and all stable/unstable manifolds of $v_{t}$, and the projection
of $K_{s}$ to $\Sigma$ via the flow of $v_{t}$ has a simple tangency.
\end{enumerate}
The monodromies of those singularities can be translated as shown
below. Note that, in the case 1, we have replaced stabilizations of
type II by special stabilizations, as a stabilization of type II can
be seen as the composition of a special stabilization followed by
a sequence of handleslides of type II, IV, and (A/$\alpha$)-isotopies. 
\begin{enumerate}
\item Commutation of $\alpha$-handleslides with stabilizations of type
II
\item Commutation of $\alpha$-handleslides with handleslides of type II
\item Commutation of $\alpha$-handleslides with compositions of stabilizations
of type II and handleslides of type III
\item Commutation of $\alpha$-handleslides with compositions of stabilizations
of type II and handleslides of type IV
\end{enumerate}
For a type 2 monodromy, Lemma \ref{associativity-1} can be used to
prove its commutativity, as in the proof of Lemma \ref{comm-lemma-stabII}.
The type 1 monodromy is clearly commutative from the construction
of maps associated to special stabilizations, which was given in the
proof of Lemma \ref{spstab-isom}. A similar argument can be applied
to monodromies of type 3 or 4.

It remains to show that $\alpha$-handleslide maps and isotopy maps
commute with isotopy maps, which correspond to the regions of $\mathcal{F}^{\prime}$
which does not contain codimension-2 singularity, which can actually
be proven by a direct application of the proof of Lemma 6.2 in \cite{Kang}.
Therefore, as in proof of naturality in \cite{Kang}, we have shown
that the diagram in Table \ref{table1} commutes after applying $\widehat{HF}_{\mathbb{Z}_{2}}$,
where $\mathcal{E}_{i}^{j}$ are extended bridge diagrams and arrows
are extended Heegaard moves. Note that leftmost column and the rightmost
column must be identical, i.e. $\mathcal{E}_{1}^{i}=\mathcal{E}_{n}^{i}$
for all $i=1,\cdots,k$, and the extended bridge diagrams $\mathcal{E}_{1}^{k},\cdots,\mathcal{E}_{n}^{k}$
in the bottom row satisfy the property that all A-arcs and B-arcs
are contained in the region, bounded by $\alpha$-curves and $\beta$-curves,
which contains the basepoint. Also, the handleslides among the arrows
in the bottom row can only be handleslides of type I. 

Since the extended bridge diagrams appearing in the top row are of
the form $\mathcal{B}\sharp\mathcal{H}$ for a genus 0 bridge diagram
$\mathcal{B}$ and a (weakly admissible) Heegaard diagram $\mathcal{H}$,
where the connected sum is taken near the basepoint $z$ of $\mathcal{H}$,
they clearly achieve hypothesis (EH-2) in \cite{eqv-Floer}, they
satisfy equivariant transversality, so that we have a natural isomorphism
\[
\widehat{HF}_{\mathbb{Z}_{2}}(\mathcal{B}\sharp\mathcal{H})\simeq\widehat{HF}_{\mathbb{Z}_{2}}(\mathcal{B})\otimes_{\mathbb{F}_{2}}\widehat{HF}(\mathcal{H})\simeq\widehat{HF}_{\mathbb{Z}_{2}}(\mathcal{B}),
\]
 since $\mathcal{H}$ represents $S^{3}$ and so $\widehat{HF}(\mathcal{H})\simeq\mathbb{F}_{2}$.
Thus, by the naturality of $\widehat{HF}_{\mathbb{Z}_{2}}$ for genus-0
bridge diagrams of based knots in $S^{3}$, we deduce that the composition
\[
\widehat{HF}_{\mathbb{Z}_{2}}(f_{1})\circ\cdots\circ\widehat{HF}_{\mathbb{Z}_{2}}(f_{n})\,:\,\widehat{HF}_{\mathbb{Z}_{2}}(H_{n})\rightarrow\widehat{HF}_{\mathbb{Z}_{2}}(H_{1})
\]
 does not depend on the choice of a sequence $f_{1,}\cdots,f_{n}$.
Therefore the restriction of $\widehat{HF}_{\mathbb{Z}_{2}}$ to $\mathcal{G}_{\alpha}$
is a functor.
\end{proof}
\begin{table}
\[
\xymatrix{H_{1}\ar[r]^{f_{1}}\ar[d] & \cdots\ar[r]^{f_{n}} & H_{n}\ar[d]\\
\mathcal{E}_{1}^{1}\ar[d] &  & \mathcal{E}_{n}^{1}\ar[d]\\
\vdots\ar[d] &  & \vdots\ar[d]\\
\mathcal{E}_{1}^{k}\ar[r] & \cdots\ar[r] & \mathcal{E}_{n}^{k}
}
\]
\caption{\label{table1}A diagram of extended bridge diagrams and extended
Heegaard moves, which becomes commutative after taking $\widehat{HF}_{\mathbb{Z}_{2}}$.
What we have actually shown is that the interior of this diagram can
be filled with smaller diagrams which represent monodromies around
the singularities of $\mathcal{F}^{\prime}$, and such diagrams are
commutative after taking $\widehat{HF}_{\mathbb{Z}_{2}}$.}
\end{table}
\begin{rem}
\label{rem:functor-ext}The proof of Lemma \ref{lem:alpha-functor}
can be directly extended to extended bridge diagrams representing
based knots in $S^{3}$ to show that any loop consisting of A-equivalences
and $\alpha$-equivalences induce the identity map in $\widehat{HF}_{\mathbb{Z}_{2}}$.
Furthermore, applying \ref{lem:alpha-beta-comm} gives us that any
loop of basic moves induce the identity map in $\widehat{HF}_{\mathbb{Z}_{2}}$.
\end{rem}

\begin{prop}
\label{prop-stab-unique}Let $\mathcal{E}_{1},\mathcal{E}_{2}$ be
extended bridge diagrams of a based knot $(K,z)$ in $S^{3}$, related
by a single stabilization (of type I or II), and let $\mathcal{E}_{1}\xrightarrow{s}\mathcal{E}_{2}$
be such a stabilization. Then the induced map 
\[
\widehat{HF}_{\mathbb{Z}_{2}}(s)\,:\,\widehat{HF}_{\mathbb{Z}_{2}}(\mathcal{E}_{2})\rightarrow\widehat{HF}(\mathcal{E}_{1})
\]
 is defined uniquely.
\end{prop}

\begin{proof}
Since the proof of type I and II are similar, we will only work out
the case of type I explicitly. Let $\Sigma$ be the Heegaard surface
of $\mathcal{H}_{\text{pt}}(\mathcal{E}_{1})$ and choose a point
$w$ which is very close to the basepoint $z$. Let $\mathcal{E}$
be the extended bridge diagram given by performing a stabilization
of type I on $\mathcal{E}_{1}$ at $w$; write the induced map as
\[
\widehat{HF}_{\mathbb{Z}_{2}}(s_{w})\,:\,\widehat{HF}_{\mathbb{Z}_{2}}(\mathcal{E})\rightarrow\widehat{HF}_{\mathbb{Z}_{2}}(\mathcal{E}_{1}).
\]
 Recall that, to construct an induced map for a stabilization of type
I performed at another point on $\Sigma$, which is not near $z$,
we have to compose $\widehat{HF}_{\mathbb{Z}_{2}}(s_{w})$ with maps
induced by $\alpha$- and $\beta$-handleslides. To prove that $\widehat{HF}_{\mathbb{Z}_{2}}(s)$
is unique, we have to prove that for any sequence of $\alpha$-handleslides
which starts from and ends at $\mathcal{E}$, the induced automorphism
of $\widehat{HF}_{\mathbb{Z}_{2}}(\mathcal{E})$ is the identity.
But this is a direct consequence of Lemma \ref{lem-funt}. Therefore
$\widehat{HF}_{\mathbb{Z}_{2}}(s)$ is uniquely defined.
\end{proof}
With the fact that the maps induced by stabilizations are uniquely
defined, we can now finish the proof of functoriality condition.
\begin{lem}
\label{lem-funt}The weak Heegaard invariant $\widehat{HF}_{\mathbb{Z}_{2}}$
satisfies functoriality condition of Definition \ref{def:stronginv}.
\end{lem}

\begin{proof}
By Lemma \ref{lem:alpha-functor}, we know that the restriction of
$\widehat{HF}_{\mathbb{Z}_{2}}$ to $\mathcal{G}_{\alpha}$ is a functor.
So, by the same argument applied to $\beta$-curves instead of $\alpha$-curves,
we also see that the restriction of $\widehat{HF}_{\mathbb{Z}_{2}}$
to $\mathcal{G}_{\beta}$ is also a functor. If $e$ is a stabilization
and $e^{\prime}$ is its reverse, we can apply Proposition \ref{prop-stab-unique}
to reduce to the case when the (de)stabilization is performed near
the basepoint, in which case $\widehat{HF}_{\mathbb{Z}_{2}}(e)\circ\widehat{HF}_{\mathbb{Z}_{2}}(e^{\prime})=\text{id}$
holds by choosing a suitable family of almost complex structures.
The diffeomorphism part is obvious. 
\end{proof}
We now move on to a proof of continuity condition.
\begin{lem}
\label{lem-cont}The weak Heegaard invariant $\widehat{HF}_{\mathbb{Z}_{2}}$
satisfies continuity condition of Definition \ref{def:stronginv}.
\end{lem}

\begin{proof}
We will prove a slightly more general statement, that for any weakly
admissible extended bridge diagram $\mathcal{E}$ of a based knot
$(K,z)$ in $S^{3}$ and a self-diffeomorphism $\phi$ of $\mathcal{E}$
which is isotopic to the identity, the map 
\[
\widehat{HF}_{\mathbb{Z}_{2}}(\phi)\in\text{Aut}_{\mathbb{F}_{2}[\theta]}(\widehat{HF}_{\mathbb{Z}_{2}}(\Sigma(K),z))
\]
 induced by $\phi$ is the identity.

If $\mathcal{E}$ is a bridge diagram on a sphere, then by the genus-0
naturality of $\widehat{HF}_{\mathbb{Z}_{2}}$, we know that $\widehat{HF}_{\mathbb{Z}_{2}}(\phi)=\text{id}$.
Next, suppose that $\mathcal{E}$ satisfies the following property.
\begin{itemize}
\item There exists a small disk $D$ on the Heegaard surface $\Sigma$ of
$\mathcal{H}_{\text{pt}}(\mathcal{E})$, near its basepoint $z$,
so that all A- and B-arcs of $\mathcal{H}_{\text{pt}}(\mathcal{E})$
are contained in $D$ and all $\alpha$- and $\beta$-arcs are contained
in $\Sigma\backslash D$.
\end{itemize}
Then we may assume that $\phi$ fixes the curve $\partial D$, and
by taking $\partial D$ as the connected-sum neck, we can represent
$\mathcal{E}$ as a connected sum:
\[
\mathcal{E}=\mathcal{B}\sharp\mathcal{H},
\]
 where $\mathcal{B}$ is a bridge diagram of $(K,z)$ on a sphere
and $\mathcal{H}$ is a Heegaard diagram representing $S^{3}$. As
in the proof of \ref{lem:alpha-functor}, we have a decomposition
\[
\widehat{HF}_{\mathbb{Z}_{2}}(\mathcal{E})\simeq\widehat{HF}_{\mathbb{Z}_{2}}(\mathcal{B})\otimes_{\mathbb{F}_{2}}\widehat{HF}(\mathcal{H},\text{pt})\simeq\widehat{HF}_{\mathbb{Z}_{2}}(\mathcal{B})
\]
 by stretching the neck to infinite length. By assumption, the action
of $\phi$ on $\widehat{HF}_{\mathbb{Z}_{2}}(\mathcal{H})$ reduces
to $\widehat{HF}_{\mathbb{Z}_{2}}(\mathcal{B})$. But since $\mathcal{B}$
is a genus-0 diagram, we already know that $\phi$ acts trivially
on it. Hence $\widehat{HF}_{\mathbb{Z}_{2}}(\phi)=\text{id}$ in this
case.

Finally, we work out the general case. As in the proof of Lemma \ref{lem:alpha-functor},
we know that there exists a sequence of extended Heegaard moves, except
stabilizations of type I (whose ``induced map'' is not uniquely
defined yet in the general case, in particular, if it occurs in a
point not close to the basepoint), from $\mathcal{E}$ to another
extended bridge diagram $\mathcal{E}^{\prime}$ which satisfies the
above property. Now, by the definition of diffeomorphism map (it just
acts by diffeomorphism), we know that it commutes with maps induced
by all extended Heegaard moves except stabilizations of type I. Therefore
the problem reduces to the former case and we are done.
\end{proof}
Now we will prove that the commutativity condition also holds. For
that, we recall the definition of distinguished rectangles, defined
in Definition 2.30 of \cite{Juhasz-naturality}.
\begin{defn}
Let $H_{i}=(\Sigma_{i},[\boldsymbol{\alpha}_{i}],[\boldsymbol{\beta}_{i}])$
be isotopy diagrams for $i=1,\cdots,4$. A distinguished rectangle
in $\mathcal{G}$ is a subgraph 
\[
\xymatrix{H_{1}\ar[r]^{e}\ar[d]_{f} & H_{2}\ar[d]^{g}\\
H_{3}\ar[r]^{h} & H_{4}
}
\]
 of $\mathcal{G}$ that satisfies one of the following properties.
\begin{enumerate}
\item Both $e$ and $h$ are $\alpha$-equivalences, while both $f$ and
$g$ are $\beta$-equivalences.
\item Both $e$ and $h$ are $\alpha$- or $\beta$-equivalences, while
$f$ and $g$ or both stabilizations.
\item Both $e$ and $h$ are $\alpha$- or $\beta$-equivalences, while
$f$ and $g$ are both diffeomorphisms. In this case, we necessarily
have $\Sigma_{1}=\Sigma_{2}$ and $\Sigma_{3}=\Sigma_{4}$, and we
require in addition that the diffeomorphisms $\Sigma_{1}\xrightarrow{f}\Sigma_{3}$
and $\Sigma_{2}\xrightarrow{g}\Sigma_{4}$ are the same.
\item The maps $e,f,g,h$ are all stabilizations, such that there are disjoint
disks $D_{1},D_{2}\subset\Sigma_{1}$ and disjoint punctured tori
$T_{1},T_{2}\subset\Sigma_{4}$ satisfying $\Sigma_{1}\backslash(D_{1}\cup D_{2})=\Sigma_{4}\backslash(T_{1}\cup T_{2})$,
$\Sigma_{2}=(\Sigma_{1}\backslash D_{1})\cup T_{1}$, and $\Sigma_{3}=(\Sigma_{1}\backslash D_{2})\cup T_{2}$.
\item The maps $e,h$ are stabilizations, while $f,g$ are diffeomorphisms.
Furthermore, there are disks $D\subset\Sigma_{1}$ and $D^{\prime}\subset\Sigma_{4}$
and punctured tori $T\subset\Sigma_{2}$ and $T^{\prime}\subset\Sigma_{4}$
such that $\Sigma_{1}\backslash D=\Sigma_{2}\backslash T$, $\Sigma_{3}\backslash D^{\prime}=\Sigma_{4}\backslash T^{\prime}$,
and the diffeomorphisms $f,g$ satisfy $f(D)=D^{\prime}$, $g(T)=T^{\prime}$,
and $f|_{\Sigma_{1}\backslash D}=g|_{\Sigma_{2}\backslash T}$.
\end{enumerate}
\end{defn}

\begin{lem}
\label{lem-comm}The weak Heegaard invariant $\widehat{HF}_{\mathbb{Z}_{2}}$
satisfies commutativity condition of Definition \ref{def:stronginv}.
\end{lem}

\begin{proof}
Commutativity for distinguished rectangles of type 1 is proven in
Lemma \ref{lem:alpha-beta-comm}. For the remaining cases, recall
that a stabilization map are defined by composing the map induced
by a stabilization near the basepoint following by $\alpha$- and
$\beta$-equivalences. Also, given a knot Heegaard diagram $\mathcal{H}$
representing a based knot $(K,z)$ and another diagram $\mathcal{H}^{\prime}$
given by stabilizing $\mathcal{H}$ near $z$, the stabilization map
is given explicitly by 
\begin{align*}
\widehat{HF}_{\mathbb{Z}_{2}}(\tilde{\mathcal{H}}) & \rightarrow\widehat{HF}_{\mathbb{Z}_{2}}(\tilde{\mathcal{H}}^{\prime}),\\
\mathbf{x} & \mapsto\mathbf{x}\cup p^{-1}(\{c\}),
\end{align*}
 where $p$ is the branched double covering map and $c$ is the new
intersection point introduced by the stabilization, by taking family
of almost complex structure of the form $\text{Sym}^{g}(\mathfrak{j})$,
the reason being that any holomorphic disk involving the Floer generator
in $p^{-1}(\{c\})$ must intersect the basepoint. But this argument
also works for counting holomorphic triangles, and thus we deduce
that the maps induced by $\alpha$- and $\beta$-equivalences and
stabilizations must commute, thereby proving commutativity of distinguished
rectangles of type 2.

Now, we can apply the functoriality condition, which is already proven
in Lemma \ref{lem-funt}, together with commutativity of distinguisehd
rectangles of type 2, to reduce the distinguished rectangles of type
3,4,5 to the case when all stabilizations involved are performed near
the basepoint. Then, by the above description of maps induced by stabilizations
near the basepoint $z$, they must commute with stabilization maps
and diffeomorphism maps. Therefore $\widehat{HF}_{\mathbb{Z}_{2}}$
satisfies commutativity for all distinguished rectangles.
\end{proof}
\begin{rem}
\label{rem:comm-stabI}The proof of Lemma \ref{lem-comm} can be directly
extended to stabilization of type I on extended bridge diagrams. The
result we get is that stabilization maps of type I and diffeomorphism
maps commutes with any maps induced by extended Heegaard diagrams.
\end{rem}

We will now prove that $\widehat{HF}_{\mathbb{Z}_{2}}$ satisfies
handleslide invariance.
\begin{lem}
\label{lem:handleswap-inv}The weak Heegaard invariant $\widehat{HF}_{\mathbb{Z}_{2}}$
satisfies handleswap invariance condition of Definition \ref{def:stronginv}.
\end{lem}

\begin{proof}
Let $H=H_{0}\sharp H_{s}$ be a knot Heegaard diagram representing
a based knot $(K,z)$ in $S^{3}$, where $H_{0}$ is double-pointed,
$H_{s}$ is the diagram drawn in Figure \ref{fig04}, and suppose
that we perform a simple handleswap on $H_{s}$. If the connected
sum is taken near $z$, then as in the proof of Lemma 9.30 in \cite{Juhasz-naturality},
we are done. 

In the general case, let $\Sigma_{0}$ and $\Sigma_{s}$ be the Heegaard
surfaces of $H_{0}$ and $H_{s}$, respectively, and let $p$ be the
point in $\Sigma_{0}$ at which we take the connected sum. Define
a subset of $\Sigma_{0}\backslash(\{z\}\cup(\alpha,\beta\text{-curves of }H_{0})$
as follows. 
\[
X=\{p\in\Sigma_{0}\backslash(\{z\}\cup(\alpha,\beta\text{-curves of }H_{0})\,\vert\,\text{handleswap invariance holds for }H\}
\]
 Then, by functoriality(Lemma \ref{lem-funt}) and continuity(Lemma
\ref{lem-comm}), we know that $X$ is a union of some connected components..
Also, we already know that the connected component containing $z$
is contained in $X$, so that $X\neq\emptyset$.

Now let $p\in\Sigma_{0}\backslash(\alpha,\beta\text{-curves of }H_{0})$
be a point satisfying the following property.
\begin{itemize}
\item Let $R$ be the connected component of $\Sigma_{0}\backslash(\{z\}\cup(\alpha,\beta\text{-curves of }H_{0}))$
containing $p$. Then $\overline{R}\cap X$ contains a segment of
either an $\alpha$-curve or a $\beta$-curve.
\end{itemize}
Then, without loss of generality, we may assume that there exists
a point $p^{\prime}\in X$ such that, when we denote the connected
sums of $H_{0}$ and $H_{s}$ performed at $p$ and $p^{\prime}$
as $H$ and $H^{\prime}$, respecively, the knot Heegaard diagrams
$H$ and $H^{\prime}$ are related by a sequence of four $\alpha$-handleslides.
Hence, by Lemma 3.25 of \cite{eqv-Floer}, the equivariant triangle
map 
\[
F_{H,H^{\prime}}\,:\,\widehat{HF}_{\mathbb{Z}_{2}}(\widetilde{H^{\prime}})\rightarrow\widehat{HF}_{\mathbb{Z}_{2}}(\widetilde{H}),
\]
 defined using the top class $\mathbf{x}$ drawn in Figure \ref{fig05},
is an isomorphism; we will call this map as the \textbf{$\alpha$-crossing
map}. We will not prove that the crossing map is the same as the composition
of four $\alpha$-handleslide maps, as we do not need it to prove
this lemma. Note that, although we are drawing knot Heegaard diagrams,
we are actually working with their branched double coverings.

Now consider the simple handleswap diagram involving $H$, and another
simple handleswap diagram involving $H^{\prime}$. Taking $\widehat{HF}_{\mathbb{Z}_{2}}$
gives the following diagram, where every edge is an isomorphism. Note
that the innermost and the outermost triangle are simple handleswap
diagrams and all crossing maps are $\alpha$-crossing maps.
\[
\xymatrix{\widehat{HF}_{\mathbb{Z}_{2}}(\widetilde{H})\ar[dr]_{\text{\ \ \ \ \ \ \ crossing}}\ar[drrr]^{\ \ \ \ \ \alpha\text{-handleslide}}\\
 & \widehat{HF}_{\mathbb{Z}_{2}}(\widetilde{H^{\prime}})\ar[r] & \widehat{HF}_{\mathbb{Z}_{2}}(\widetilde{H_{1}^{\prime}})\ar[dl] & \widehat{HF}_{\mathbb{Z}_{2}}(\widetilde{H_{1})}\ar[l]^{\text{crossing}}\ar[ddll]^{\ \ \ \beta\text{-handleslide}}\\
 & \widehat{HF}_{\mathbb{Z}_{2}}(\widetilde{H_{2}^{\prime}})\ar[u]\\
 & \widehat{HF}_{\mathbb{Z}_{2}}(\widetilde{H_{2}})\ar[uuul]^{\text{diffeo}}\ar[u]_{\text{crossing}}
}
\]
Since diffeomorphism maps clearly commutes with crossing maps, the
leftmost face is commutative. Also, by Lemma \ref{associativity-2},
we see that the face in the right-bottom corner is also commutative,
and the same argument works for the commutativity of the topmost face. 

Now, the central triangle is commutive by assumption. Thus every face
of the diagram is commutative. Since all edges of the diagram are
isomorphisms, the peripheral triangle must also be commutative. Hence
$p\in X$ by the definition of $X$. Therefore, by induction on the
number of $\alpha$- and $\beta$-curves needed to cross to travel
from $p$ to $z$, we deduce that 
\[
X=\Sigma_{0}\backslash(\{z\}\cup(\alpha,\beta\text{-curves of }H_{0}),
\]
 i.e. $\widehat{HF}_{\mathbb{Z}_{2}}$ satisfies handleswap invariance
condition of Definition \ref{def:stronginv}.
\end{proof}
\begin{figure}
\resizebox{.4\textwidth}{!}{\includegraphics{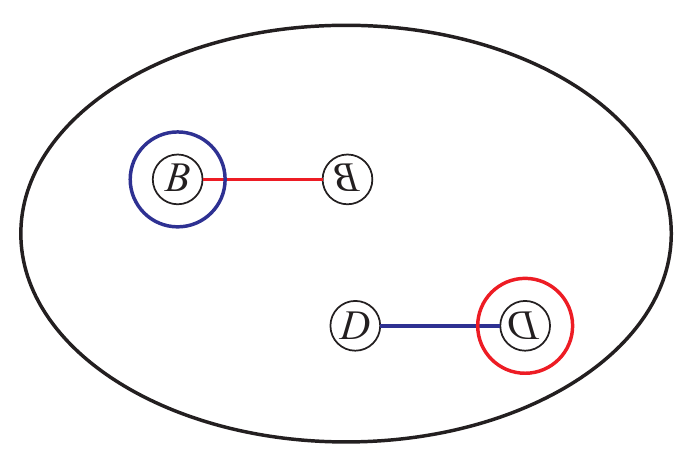}} \caption{\label{fig04}The Heegaard diagram $H_{s}$; the boundary is collapsed
to form a genus 2 surface.}
\end{figure}

\begin{figure}
\resizebox{.4\textwidth}{!}{\includegraphics{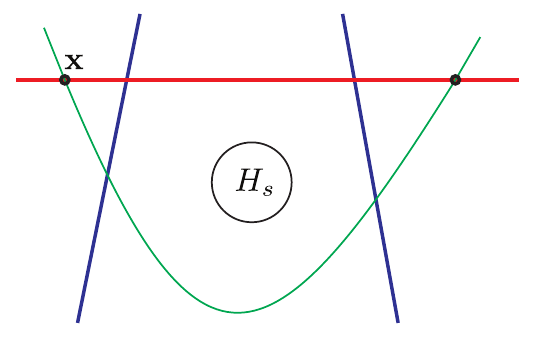}} \caption{\label{fig05}The triple-diagram near the connected sum region. The
circle in the middle is where the diagram $H_{s}$ is attached via
connected sum neck.}
\end{figure}
\begin{thm}
\label{thm:stronginv} There exists a strong Heegaard invariant of
based knots in $S^{3}$, which is isomorphic to $\widehat{HF}_{\mathbb{Z}_{2}}$.
\end{thm}

\begin{proof}
By Lemma \ref{lem-funt}, Lemma \ref{lem-cont}, Lemma \ref{lem-comm},
and Lemma \ref{lem:handleswap-inv}, $\widehat{HF}_{\mathbb{Z}_{2}}$
satisfies all conditions of Definition \ref{def:stronginv}. Therefore
it is a strong Heegaard invariant.
\end{proof}
The important point of Theorem \ref{thm:stronginv} is that, instead
of treating $\widehat{HF}$ directly as a Heegaard invariant, we have
used the word ``isomorphic''. The reason is as follows. While it
is true that we have constructed a strong Heegaard invariant $\widehat{HF}_{\mathbb{Z}_{2}}(\Sigma(K),z)$
of based knots $(K,z)$ in $S^{3}$, defined using knot Heegaard diagrams,we
also have another definition of $\widehat{HF}_{\mathbb{Z}_{2}}(\Sigma(K),z)$
using bridge diagrams of $(K,z)$ on $S^{2}$. Both versions of $\widehat{HF}_{\mathbb{Z}_{2}}$
are natural invariants, i.e. define functors 
\[
\widehat{HF}_{\mathbb{Z}_{2}}^{knot},\,\widehat{HF}_{\mathbb{Z}_{2}}^{bridge}\,:\,\text{Knot}_{\ast}\rightarrow\text{Mod}_{\mathbb{F}_{2}[\theta]},
\]
 where $\text{Knot}_{\ast}$ is the category whose objects are based
knots in $S^{3}$ and morphisms are diffeomorphisms. Also, Theorem
\ref{thm:stronginv} tells us that both of the have the same isomorphism
type (of $\mathbb{F}_{2}[\theta]$-modules). However, we have not
yet proven whether they are isomorphic through a natural isomorphism,
i.e. there exists a natural transformation $\eta$ between functors
$\widehat{HF}_{\mathbb{Z}_{2}}^{knot}$ and $\widehat{HF}_{\mathbb{Z}_{2}}^{bridge}$. 

We will now construct such a natural transformation; note that, while
dealing with this issue, we will strictly distinguish the two invariants
$\widehat{HF}_{\mathbb{Z}_{2}}^{knot}$ and $\widehat{HF}_{\mathbb{Z}_{2}}^{bridge}$.
Let $\mathcal{E}=((\Sigma,\boldsymbol{\alpha},\boldsymbol{\beta}),(P,A,B,z))$
be an extended bridge diagram representing a based knot $(K,z)$ in
$S^{3}$. Choose a point $p\in\text{int}(A_{i})\cap\text{int}(B_{j})$
for some $A_{i}\in A$, $B_{j}\in B$, where $z\notin A_{i}\cup B_{j}$;
such a point will be called as a \textbf{proper crossing} of $\mathcal{E}$.
Let $D\subset\Sigma$ be a small disk neighborhood of $p$, which
does not intersect $\alpha$- and $\beta$-curves and intersects A-
and B-arcs only with $A_{i}$ and $B_{j}$, satisfying $D\cap(A_{i}\cup B_{j})=\{p\}$.
Then, on a puncture torus $T$ whose boundary $\partial T$ is identified
with $\partial D$, draw two disjoint simple arcs, denoted $A_{i}^{p}$
and $B_{j}^{p}$, so that $\partial A_{i}^{p}=\partial A_{i}$ and
$\partial B_{j}^{p}=\partial B_{j}$. Also, draw two disjoint simple
closed curves, denoted $\alpha_{p}$ and $\beta_{p}$, so that the
following conditions hold.
\begin{itemize}
\item $A_{i}^{p}\cap\alpha_{p}=B_{j}^{p}\cap\beta_{p}=\emptyset$.
\item $\alpha_{p}$ intersects $B_{j}^{p}$ transversely at one point.
\item $\beta_{p}$ intersects $A_{i}^{p}$ transversely at one point.
\end{itemize}
Now define the following objects.
\begin{itemize}
\item $\Sigma^{p}=(\Sigma\backslash D)\cup T$
\item $\boldsymbol{\alpha}^{p}=\boldsymbol{\alpha}\cup\{\alpha_{p}\}$,
$\boldsymbol{\beta}^{p}=\boldsymbol{\beta}\cup\{\beta_{p}\}$
\item $A^{p}=(A\backslash\{A_{i}\})\cup\{A_{i}^{p}\}$, $B^{p}=(B\backslash\{B_{j}\})\cup\{B_{j}^{p}\}$
\end{itemize}
Then the pair 
\[
R_{p}(\mathcal{E})=((\Sigma^{p},\boldsymbol{\alpha}^{p},\boldsymbol{\beta}^{p}),(P,A^{p},B^{p},z))
\]
 also represents $(K,z)$. Furthermore, $\mathcal{E}$ and $R_{p}(\mathcal{E})$
are related by a sequence consisting of a stabilization of type I
and two handleslides of type III. Hence, by taking $\widehat{HF}_{\mathbb{Z}_{2}}$
and composing the induced maps, we get an isomorphism, which we will
call as the resolution map at $p$:
\[
\mathcal{R}_{p,\mathcal{E}}\,:\,\widehat{HF}_{\mathbb{Z}_{2}}(R_{p}(\mathcal{E}))\rightarrow\widehat{HF}_{\mathbb{Z}_{2}}(\mathcal{E}).
\]
\begin{lem}
\label{lem:resmap-welldef}Given an extended bridge diagram $\mathcal{E}$
representing a based knot $(K,z)$ in $S^{3}$, and its proper crossing
$p$, the resolution map $\mathcal{R}_{p,\mathcal{E}}$ depends only
on the choice of $p$.
\end{lem}

\begin{proof}
The definition of $\mathcal{R}_{p,\mathcal{E}}$ involves choosing
a point in $D\backslash(A_{i}\cup B_{j})$, at which a stabilization
of type I would occur, and also choosing which one of two handleslides
of type I/II is to be applied first. Take any two possible choices,
and denote the two maps given by taking compositions of the induced
maps as $\mathcal{R}_{p,\mathcal{E}}^{1}$ and $\mathcal{R}_{p,\mathcal{E}}^{2}$.
Then the composition $(\mathcal{R}_{p,\mathcal{E}}^{2})^{-1}\circ\mathcal{R}_{p,\mathcal{E}}^{1}$,
which is an automorphism of $\widehat{HF}_{\mathbb{Z}_{2}}(\Sigma(K),z)$,
can be decomposed as 
\begin{align*}
(\mathcal{R}_{p,\mathcal{E}}^{2})^{-1}\circ\mathcal{R}_{p,\mathcal{E}}^{1} & =(\text{basic moves})\circ(\text{destabilization})\circ(\text{stabilization})\circ(\text{basic moves})\\
 & =(\text{basic moves})\circ(\text{special destab})\circ(\text{special stab})\circ(\text{basic moves})\\
 & =\text{composition of basic moves}
\end{align*}
 where Lemma \ref{lem:alpha-beta-comm} is applied in the first line.
Now, by Remark \ref{rem:functor-ext}, this map should be the identity.
Therefore $\mathcal{R}_{p,\mathcal{E}}^{1}=\mathcal{R}_{p,\mathcal{E}}^{2}$.
\end{proof}
\begin{lem}
\label{lem:resmap-comm}Let $p,q$ be two distinct proper crossings
of $\mathcal{E}$. Then we have 
\[
\mathcal{R}_{p,\mathcal{E}}\circ\mathcal{R}_{q,R_{p}(\mathcal{E})}=\mathcal{R}_{q,\mathcal{E}}\circ\mathcal{R}_{p,R_{q}(\mathcal{E})}.
\]
 In other words, resolution maps of distinct proper crossings commute.
\end{lem}

\begin{proof}
Using Lemma \ref{lem-comm} and Remark \ref{lem:resmap-welldef},
we can see that the map 
\[
(\mathcal{R}_{p,\mathcal{E}}\circ\mathcal{R}_{q,R_{p}(\mathcal{E})})^{-1}\circ\mathcal{R}_{q,\mathcal{E}}\circ\mathcal{R}_{p,R_{q}(\mathcal{E})}
\]
 is a composition of maps induced by a loop of basic moves, thus is
the identity.
\end{proof}
Given an extended bridge diagram $\mathcal{E}$ representing a based
knot $(K,z)$ in $S^{3}$, let $\mathcal{I}$ be the set of all proper
crossings of $\mathcal{E}$. Choose an enumeration $\mathcal{I}=\{x_{1},\cdots,x_{n}\}$.
Define an extended bridge diagram $R(\mathcal{E})$, which depends
only on $\mathcal{E}$, as follows. 
\[
R(\mathcal{E})=R_{x_{n-1}}(\cdots(\mathcal{R}_{x_{1}}(\mathcal{E}))\cdots)
\]
 Also, define an isomorphism $\mathcal{R}_{\mathcal{E}}\,:\,\widehat{HF}_{\mathbb{Z}_{2}}(R(\mathcal{E}))\rightarrow\widehat{HF}_{\mathbb{Z}_{2}}(\mathcal{E})$
as follows.
\[
\mathcal{R}_{\mathcal{E}}=\mathcal{R}_{x_{1},\mathcal{E}}\circ\cdots\circ\mathcal{R}_{x_{n},R_{x_{n-1}}(\cdots(\mathcal{R}_{x_{1}}(\mathcal{E}))\cdots)}
\]
 Then, by Lemma \ref{lem:resmap-welldef} and Lemma \ref{lem:resmap-comm},
$\mathcal{R}_{\mathcal{E}}$ does not depend on the choice of an enumeration
of $\mathcal{I}$, thus depends only on $\mathcal{E}$.

Now observe that the A- and B-arcs of $R(\mathcal{E})$ which do not
contain $z$ also do not intersect themselves in their interior. So,
after performing handleslides as in Figure \ref{fig06} on $R(\mathcal{E})$,
it can then be destabilized (of type II) to a knot Heegaard diagram
$HR(\mathcal{E})$ representing $(K,z)$. Denote the composition of
destabilization maps as 
\[
\mathcal{D}_{\mathcal{E}}\,:\,\widehat{HF}_{\mathbb{Z}_{2}}(\widetilde{HR(\mathcal{E})})\rightarrow\widehat{HF}_{\mathbb{Z}_{2}}(R(\mathcal{E})).
\]
 Then, by Lemma \ref{comm-lemma-stabII}, $\mathcal{D}_{\mathcal{E}}$
also depends only on $\mathcal{E}$. 

\begin{figure}
\resizebox{.4\textwidth}{!}{\includegraphics{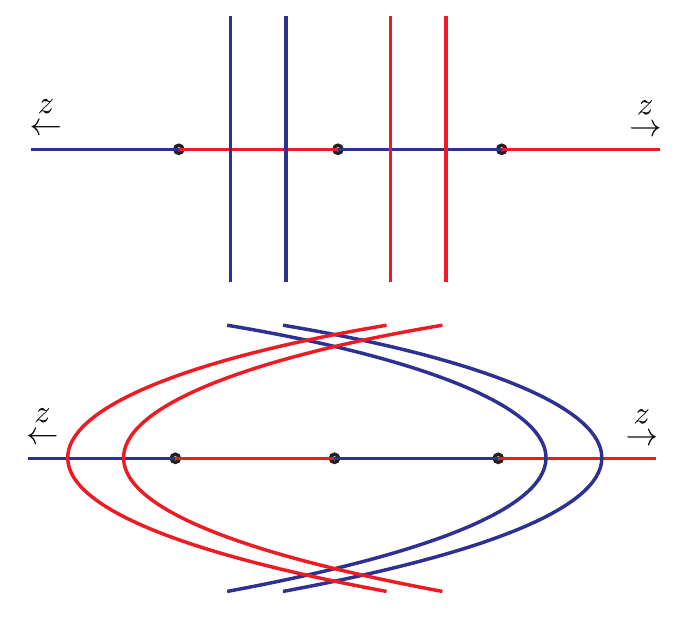}} \caption{\label{fig06}The diagrams $R(\mathcal{E})$(above) and $HR(\mathcal{E})$(below)
near the A-arcs and B-arcs. Here, we assume for simplicity that the
rightmost A-arc and the leftmost B-arc are the ones connected to the
basepoint $z$.}
\end{figure}
\begin{defn}
Given an extended bridge diagram $\mathcal{E}$ representing a based
knot $(K,z)$ in $S^{3}$, we define its translation isomorphism as
\[
\mathcal{T}_{\mathcal{E}}=(\mathcal{R}_{\mathcal{E}}\circ\mathcal{D}_{\mathcal{E}})^{-1}\,:\widehat{HF}_{\mathbb{Z}_{2}}(\mathcal{E})\rightarrow\,\widehat{HF}_{\mathbb{Z}_{2}}(\widetilde{HR(\mathcal{E})}).
\]
\end{defn}

We will now prove that translation isomorphisms define a natural transformation
between $\widehat{HF}_{\mathbb{Z}_{2}}^{knot}$ and $\widehat{HF}_{\mathbb{Z}_{2}}^{bridge}$.
\begin{lem}
\label{lem:HR-relations}Let $\mathcal{E}\xrightarrow{s}\mathcal{E}^{\prime}$
be an extended Heegaard move between extended bridge diagrams $\mathcal{E},\mathcal{E}^{\prime}$
representing a based knot $(K,z)$ in $S^{3}$. 
\begin{enumerate}
\item If $s$ is a basic move, then the knot Heeegaard diagrams $HR(\mathcal{E})$
and $HR(\mathcal{E}^{\prime})$ are related by a sequence of basic
moves and stabilizations of type I. 
\item If $s$ is a stabilization of type I, then $HR(\mathcal{E})$ and
$HR(\mathcal{E}^{\prime})$ are also related by a stabilization of
type I.
\item If $s$ is a stabilization of type II, then $HR(\mathcal{E})=HR(\mathcal{E}^{\prime})$.
\item If $s$ is a diffeomorphism, then $HR(\mathcal{E})$ and $HR(\mathcal{E}^{\prime})$
are also related by a diffeomorphism.
\end{enumerate}
\end{lem}

\begin{proof}
Cases 2, 3, and 4 are obvious. For case 1, it can be easily seen that
$R(\mathcal{E})$ and $R(\mathcal{E}^{\prime})$ are related by a
sequence of basic moves and stabilizations of type I. Since $HR$
is obtained from $R$ by performing handleslides of type III, we deduce
that $HR(\mathcal{E})$ and $HR(\mathcal{E}^{\prime})$ are also related
by a sequecne of basic moves and stabilizations of type I.
\end{proof}
\begin{thm}
Translation isomorphisms define an invertible natural transformation
between the functors 
\[
\widehat{HF}_{\mathbb{Z}_{2}}^{knot},\,\widehat{HF}_{\mathbb{Z}_{2}}^{bridge}\,:\,\text{Knot}_{\ast}\rightarrow\text{Mod}_{\mathbb{F}_{2}[\theta]}.
\]
\end{thm}

\begin{proof}
Let $\mathcal{B}\xrightarrow{s}\mathcal{B}^{\prime}$ be an extended
Heegaard diagram between bridge diagrams $\mathcal{B},\mathcal{B}^{\prime}$
representing a based knot $(K,z)$ in $S^{3}$, and denote the isomorphism
induced by $s$ by 
\[
\widehat{HF}_{\mathbb{Z}_{2}}^{bridge}(s)\,:\,\widehat{HF}_{\mathbb{Z}_{2}}^{bridge}(\mathcal{B}^{\prime})\rightarrow\widehat{HF}_{\mathbb{Z}_{2}}^{bridge}(\mathcal{B}).
\]
If $s$ is a basic move or a stabilization, then by Lemma \ref{lem:HR-relations},
there exists either a (possibly empty) sequence of basic moves and
stabilizations from $HR(\mathcal{B})$ and $HR(\mathcal{B}^{\prime})$;
denote the induced map by 
\[
\widehat{HF}_{\mathbb{Z}_{2}}^{knot}(HR(s))\,:\,\widehat{HF}_{\mathbb{Z}_{2}}^{bridge}(\widetilde{HR(\mathcal{B}^{\prime})})\rightarrow\widehat{HF}_{\mathbb{Z}_{2}}^{bridge}(\widetilde{HR(\mathcal{B})}).
\]
 Then $(\widehat{HF}_{\mathbb{Z}_{2}}^{bridge}(s)\circ\mathcal{T}_{\mathcal{B}})^{-1}\circ\mathcal{T}_{\mathcal{B}^{\prime}}\circ\widehat{HF}_{\mathbb{Z}_{2}}^{knot}(HR(s))$
is the map induced by a loop consisting of basic moves and stabilizations.
Hence, by Lemma \ref{comm-lemma-stabII}, Remark \ref{rem:comm-stabI},
and Remark \ref{rem:functor-ext}, we see that it is equal to the
identity map, i.e. 
\[
\widehat{HF}_{\mathbb{Z}_{2}}^{bridge}(s)\circ\mathcal{T}_{\mathcal{B}}=\mathcal{T}_{\mathcal{B}^{\prime}}\circ\widehat{HF}_{\mathbb{Z}_{2}}^{knot}(HR(s)).
\]
 Thus the translation maps induce a well-defined map 
\[
\mathcal{T}_{(K,z)}\,:\,\widehat{HF}_{\mathbb{Z}_{2}}^{bridge}(\Sigma(K),z)\rightarrow\widehat{HF}_{\mathbb{Z}_{2}}^{knot}(\Sigma(K),z).
\]
 Now, again by Remark \ref{rem:comm-stabI}, we know that $\mathcal{T}_{K}$
commutes with diffeomorphism maps. However the morphisms of the category
$\text{Knot}_{\ast}$ are precisely diffeomorphism. Therefore the
correspondence 
\[
\mathcal{T}\,:\,(K,z)\mapsto\mathcal{T}_{(K,z)}
\]
 is a natural transformation from $\widehat{HF}_{\mathbb{Z}_{2}}^{bridge}$
to $\widehat{HF}_{\mathbb{Z}_{2}}^{knot}$. Since $\mathcal{T}_{(K,z)}$
is an isomorphism for each based knot $(K,z)$ in $S^{3}$, the natural
transformation $\mathcal{T}$ is invertible. 
\end{proof}

\section{The $\widehat{HF}_{\mathbb{Z}_{2}}$ of very nice knot Heegaard diagrams}

Given a knot $K$ in $S^{3}$, we can combinatorially compute its
knot Floer homology $\widehat{HFK}(S^{3},K)$ in the following way.
Choose any knot Heegaard diagram $H_{0}$ representing $K$. Theorem
1.2 of \cite{Sarkar-Wang} tells us that we can convert $H_{0}$ into
a nice diagram $H$ using only isotopies and handleslides. Here, we
say that $H$ is nice if the follwing condition holds.
\begin{itemize}
\item Every region of $H$ bounded by $\alpha$- and $\beta$-curves, which
does not contain basepoints of $H$, is either a bigon or a square.
\end{itemize}
Then, by Theorem 3.3 and Theorem 3.4 of \cite{Sarkar-Wang}, for any
Floer generators $\mathbf{x}$,$\mathbf{y}$ and a Whitney disk $\phi\in\pi_{2}^{0}(\mathbf{x},\mathbf{y})$
with $\mu(\phi)=1$, there exists a holomorphic representative of
$\phi$ if and only if the domain $\mathcal{D}(\phi)$ of $\phi$
is either a bigon or a square which does not contain basepoints of
$H$, and the moduli space of holomorphic representatives of $\phi$
is a point. This allows us to describe $\widehat{CFK}(S^{3},K)$ and
thus compute $\widehat{HFK}(S^{3},K)$ in a combinatorial way.

We will now prove that we can also compute $\widehat{HF}_{\mathbb{Z}_{2}}(\Sigma(K),z)$
in a combinatorial way, without using desirable bridge diagrams of
$K$ as in Figure 1 of \cite{eqv-Floer}. The key idea is to use diagrams
satisfying the conditions similar to nice diagrams. For the technical
terms arising in Lemma \ref{lem:techlemma} involving details of the
construction of $\mathbb{Z}_{2}$-equivariant Heegaard Floer cohomology,
one can find their definitions in section 3 of \cite{eqv-Floer}.
\begin{defn}
A knot Heegaard diagram $H=(\Sigma,\boldsymbol{\alpha},\boldsymbol{\beta},z,w)$
representing a based knot $(K,z)$ in $S^{3}$ is \textbf{very nice}
if the following conditions are satisfied.
\begin{itemize}
\item $H$ is a nice diagram.
\item The region of $H$ containing $w$ is a bigon.
\end{itemize}
\end{defn}

\begin{lem}
\label{lem:techlemma}Let $H=(\Sigma,\boldsymbol{\alpha},\boldsymbol{\beta},z,w)$
be a very nice knot Heegaard diagram of a based knot $(K,z)$ in $S^{3}$,
and let $\hat{H}=(\tilde{\Sigma},\tilde{\boldsymbol{\alpha}},\tilde{\boldsymbol{\beta}},z)$
be the diagram obtained by taking branched double cover of $H$ along
$\{z,w\}$ and forgetting $w$. Then there exists a $\mathbb{Z}_{2}$-equivariant
homotopy coherent diagram $F\,:\,\mathscr{E}\mathbb{Z}_{2}\rightarrow\overline{\mathcal{J}}$
of eventually cylindrical almost complex structures on $\text{Sym}^{g}(\tilde{\Sigma})$,
where $g$ is the genus of $\tilde{\Sigma}$, satisfying the following
conditions.
\begin{itemize}
\item For every object $a\in\text{Ob}(\mathscr{E}\mathbb{Z}_{2})$, every
pair of generators $x,y\in\mathbb{T}_{\tilde{\boldsymbol{\alpha}}}\cap\mathbb{T}_{\tilde{\boldsymbol{\beta}}}$,
and every $\phi\in\pi_{2}(x,y)$, the moduli space $\mathcal{M}(\phi;F(a))$
is transversely cut out.
\item For every composable sequence $f_{n},\cdots,f_{1}$ of morphisms of
$\mathscr{E}\mathbb{Z}_{2}$ with $n\ge2$, every pair of generators
$x,y\in\mathbb{T}_{\tilde{\boldsymbol{\alpha}}}\cap\mathbb{T}_{\tilde{\boldsymbol{\beta}}}$,
and every $\phi\in\pi_{2}(x,y)$ with $\mu(\phi)=1-n$, whose domain
does not intersect $z$, we have 
\[
\mathcal{M}(\phi;F(f_{n},\cdots,f_{1}))=\emptyset.
\]
\end{itemize}
Also, the same statment is true for the $\mathbb{Z}_{2}$-invariant
Heegaard diagram $\mathcal{H}_{d}(\widetilde{H})$.
\end{lem}

\begin{proof}
Choose any $\mathbb{Z}_{2}$-invariant complex structure $\mathfrak{j}_{0}$
on $\tilde{\Sigma}$. Then, by Theorem 3.4 of \cite{Sarkar-Wang},
a generic perturbation of $\alpha$- and $\beta$-curves achieves
transversality of moduli spaces $\mathcal{M}(\phi;\text{Sym}^{g}(\mathfrak{j}_{0}))$
for any bigons or squares $\phi$ which does not intersect $z$. If
the perturbation is sufficiently small, then it can be seen as an
action of a self-diffeomorphism of $\tilde{\Sigma}$ which is sufficiently
close to the identity. Thus there exists a 1-parameter family of complex
structure $\mathfrak{j}$, isotopic to $\mathfrak{j}_{0}$, such that
$\mathcal{M}(\phi;\text{Sym}^{g}(\mathfrak{j}))$ achieves transversality
for any bigons or squares $\phi$ which does not intersect $z$.

Let $\sigma$ denote the deck transformation of the branched double
covering $\tilde{\Sigma}\rightarrow\Sigma$. Since $\mathfrak{j}_{0}$
is $\mathbb{Z}_{2}$-invariant and $\mathfrak{j}$ is sufficiently
close to $\mathfrak{j_{0}}$, $\sigma\mathfrak{j}$ is also close
to $\mathfrak{j}_{0}$. So we can choose a generic $\mathbb{Z}_{2}$-coherent
homotopy coherent diagram $F$ of eventually cylindrical almost complex
structures on $\tilde{\Sigma}$, consisting only of families of complex
structures of the form $\text{Sym}^{g}(\mathfrak{j}_{s})$ for complex
structures $\mathfrak{j}_{s}$ on $\tilde{\Sigma}$ which are $C^{\infty}$-close
to $\mathfrak{j}_{0}$, so that the first condition is satisfied.
Also, since the moduli space of complex structures on $\tilde{\Sigma}$
has dimension $6g-6$ and thus finite-dimensional, we may assume that
every element of those families is isotopic to $\mathfrak{j}_{0}$.
Then such complex structures are pullbacks of $\mathfrak{j}_{0}$
under diffeomorphisms isotopic to the identity, so elements of $\mathcal{M}(\phi;F(f_{n},\cdots,f_{1}))$
can be seen as a holomorphic disks in $\text{Sym}^{g}(\tilde{\Sigma})$
with dynamic boundary conditions. Such disks must intersect positively
with diagonals $\{p\}\times\text{Sym}^{g-1}(\tilde{\Sigma})$ by holomorphicity
when $p$ is not contained in a neighborhood of $\alpha$- and $\beta$-curves.
Hence, if $\mathcal{M}(\phi;F(f_{n},\cdots,f_{1}))\neq\emptyset$,
the domain $\mathcal{D}(\phi)$ of $\phi$ must be nonzero and positive.

Now, since $H$ is assumed to be very nice, $\hat{H}$ is a nice diagram.
Thus, by the proof of Theorem 3.3 of \cite{Sarkar-Wang}, positivity
of $\mathcal{D}(\phi)$ and the condition $\mu(\phi)=1-n<0$ implies
that $\phi$ is either a bigon or a square, whose domain does not
intersect $z$. However, such domains have Maslov index zero, so we
get a contradiction. Furthermore, since the domains of $\mathcal{H}_{d}(\widetilde{H})$
are in 1-1 correspondence, preserving Maslov indices, with domains
of $\hat{H}$, the same conclusion holds for $\mathcal{H}_{d}(\widetilde{H})$.
\end{proof}
\begin{lem}
\label{lem:RHom-verynice}Let $H$ be a very nice knot Heegaard diagram
which represents a based knot $(K,z)$ in $S^{3}$. Then we have 
\[
\widehat{HF}_{\mathbb{Z}_{2}}(\Sigma(K),z)\simeq H^{\ast}(\text{RHom}_{\mathbb{F}_{2}[\mathbb{Z}_{2}]}(\widehat{CF}(\hat{H}),\mathbb{F}_{2}))\simeq H^{\ast}(\text{RHom}_{\mathbb{F}_{2}[\mathbb{Z}_{2}]}(\widehat{CF}(\mathcal{H}_{d}(\widetilde{H})),\mathbb{F}_{2})).
\]
\end{lem}

\begin{proof}
Let $F$ be a $\mathbb{Z}_{2}$-equivariant homotopy coherent diagram
arising in Lemma \ref{lem:techlemma}. Then all higher terms of the
differential of freed Floer complex $\widetilde{CF}_{\mathbb{Z}_{2}}(\mathcal{H}_{d}(\widetilde{H});F)$
vanish by construction, and thus we get 
\[
\widehat{CF}_{\mathbb{Z}_{2}}(\Sigma(K),z)\simeq\widehat{CF}_{\mathbb{Z}_{2}}(\mathcal{H}_{d}(\widetilde{H}))\simeq\text{RHom}_{\mathbb{F}_{2}[\mathbb{Z}_{2}]}(\widehat{CF}(\mathcal{H}_{d}(\widetilde{H})),\mathbb{F}_{2}).
\]
 Also, since $\mathcal{H}_{d}(\widetilde{H})$ is given by performing
a $\mathbb{Z}_{2}$-invariant stabilization to the Heegaard diagram
$\hat{H}$, the stabilization map 
\[
\widehat{CF}(\hat{H})\rightarrow\widehat{CF}(\mathcal{H}_{d}(\widetilde{H}))
\]
 is a $\mathbb{Z}_{2}$-equivariant quasi-isomorphism for a suitable
choice of almost complex structures. Therefore, by taking cohomology,
we get the desired result.
\end{proof}
Now we will briefly prove that any based knot can be represented by
a very nice knot Heegaard diagram by using Sarkar-Wang algorithm in
an equivariant way.
\begin{lem}
\label{lem:eqvSW}Every knot Heegaard diagram $H=(\Sigma,\boldsymbol{\alpha},\boldsymbol{\beta},z,w)$
representing a based knot $(K,z)$ in $S^{3}$ is related to a very
nice diagram by a sequence of isotopies and handleslides.
\end{lem}

\begin{proof}
Let $\hat{H}$ be the branched double cover of $H$ along the branching
locus $\{z,w\}$. Then, in each stage of Sarkar-Wang algorithm for
$\hat{H}$, we can perform the algorithm in a $\mathbb{Z}_{2}$-equivariant
manner, as in Figure \ref{fig07}. The complexity argument works in
the same way, so that the result we get after performing the algorithm
is a branched double cover $\hat{H}^{\prime}$ of some knot Heegaard
diagram $H^{\prime}$, representing $(K,z)$, such that every region
of $\hat{H}^{\prime}$ which does not contain $z$ is either a bigon
or a square. But then the region of $H^{\prime}$ which contains $w$
must be a bigon. Therefore $H^{\prime}$ is a very nice diagram.
\end{proof}
\begin{figure}
\resizebox{.4\textwidth}{!}{\includegraphics{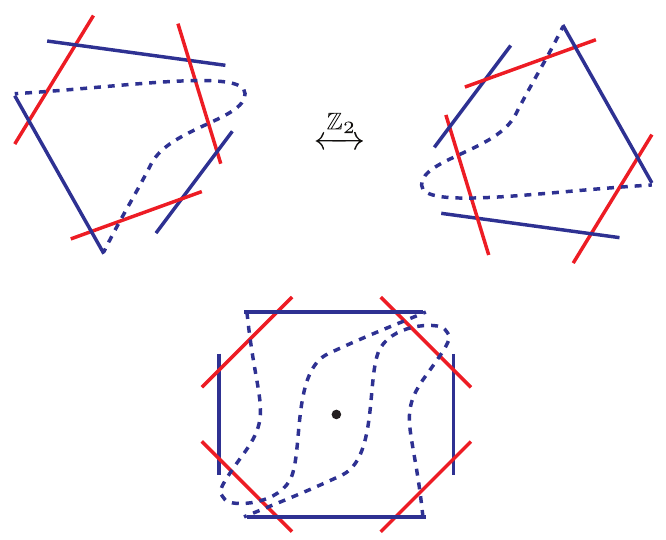}} \caption{\label{fig07}The initial stage of $\mathbb{Z}_{2}$-equivariant Sarkar-Wang
algorithm. The diagram above is the case when the algorithm starts
from a pair of bad regions which form a $\mathbb{Z}_{2}$-orbit, and
the diagram below is the case when algorithm starts from a $\mathbb{Z}_{2}$-invariant
bad region. The black point in the center of the latter is either
$z$ or $w$.}
\end{figure}

To sum up, we have proved the following theorem.
\begin{thm}
\label{thm:combinatorial}There is a combinatorial way to compute
$\widehat{HF}_{\mathbb{Z}_{2}}(\Sigma(K),z)$ of based knot $(K,z)$,
using knot Heegaard diagrams.
\end{thm}

\begin{proof}
Choose any knot Heegaard diagram $H_{0}$ representing $(K,z)$. Then,
by Lemma \ref{lem:eqvSW}, we can replace it by a very nice diagram
$H$. By Lemma \ref{lem:RHom-verynice}, one can compute $\widehat{HF}_{\mathbb{Z}_{2}}(\Sigma(K),z)$
using $\widehat{CF}(\hat{H})$. But the chain complex $\widehat{CF}(\hat{H})$
can be combinatorially computed by Theorem 3.3 and Theorem 3.4 of
\cite{Sarkar-Wang}. This gives a combitorial description of $\widehat{HF}_{\mathbb{Z}_{2}}(\Sigma(K),z)$.
\end{proof}
Finally, recall that any (weakly admissible) knot Heegaard diagram
$H=(\Sigma,\boldsymbol{\alpha},\boldsymbol{\beta},z,w)$ representing
a based knot $(K,z)$ in $S^{3}$, we have defined $\hat{H}$ as the
$\mathbb{Z}_{2}$-equivariant Heegaard diagram defined by taking branched
double cover of $H$ along $\{z,w\}$ and forgetting $w$, and defined
$\widetilde{H}$ as the extended bridge diagram obtained from $H$
by forgetting $w$ and adding an A-arc and a B-arc which start at
$w$ and do not intersect each other in their interior. In the previous
section, we have proved that 
\[
\widehat{HF}_{\mathbb{Z}_{2}}(\Sigma(K),z)\simeq\widehat{HF}_{\mathbb{Z}_{2}}(\mathcal{H}_{d}(\widetilde{H})).
\]
 Now, using Theorem \ref{thm:combinatorial}, we can remove the pair
of an A-arc and a B-arc from $\widetilde{H}$.
\begin{thm}
\label{thm:knotdiagramHF}Let $H$ be any weakly admissible knot Heegaard
diagram representing a based knot $(K,z)$ in $S^{3}$. Then we have
\[
\widehat{HF}_{\mathbb{Z}_{2}}(\Sigma(K),z)\simeq\widehat{HF}_{\mathbb{Z}_{2}}(\hat{H}).
\]
\end{thm}

\begin{proof}
By \ref{lem:eqvSW}, $H$ is related by a sequence of isotopies and
handleslides to a very nice diagram $H_{0}$. Then $\widehat{H_{0}}$
is related by a sequence of equivariant isotopies and handleslides
to $\hat{H}$. Hence, by Lemma 3.24 and Lemma 3.25 of \cite{eqv-Floer},
we have 
\[
\widehat{HF}_{\mathbb{Z}_{2}}(\hat{H})\simeq\widehat{HF}_{\mathbb{Z}_{2}}(\widehat{H_{0}}).
\]
 But since $H_{0}$ is very nice, $\widehat{HF}_{\mathbb{Z}_{2}}(\widehat{H_{0}})\simeq\widehat{HF}_{\mathbb{Z}_{2}}(\Sigma(K),z)$
by \ref{lem:RHom-verynice}.
\end{proof}

\section{$\mathbb{Z}_{2}$-equivariant knot Floer cohomology}

Theorem \ref{thm:knotdiagramHF} tells us that, given an admissible
knot Heegaard diagram $H$ of a based knot $(K,z)$ in $S^{3}$, we
can calculate $\widehat{HF}_{\mathbb{Z}_{2}}(\Sigma(K),z)$ by taking
branched double cover of $H$, removing its second basepoint, and
then taking $\widehat{HF}_{\mathbb{Z}_{2}}$ of the $\mathbb{Z}_{2}$-invariant
diagram we obtain. 

Now suppose that we do not remove the second basepoint. Then the diagram
we obtain is a knot Heegaard diagram, which now represents $K$ in
the branched double cover $\Sigma(K)$. we will now observe that taking
$\widehat{HF}_{\mathbb{Z}_{2}}$ to this diagram gives a knot invariant.
Note that, to achieve more generality, we work with links in $S^{3}$,
instead of knots.
\begin{defn}
Given a link Heegaard diagram $H$ representing a link $L$ in $S^{3}$,
we denote by $B(H)$ the Heegaard diagram representing $L$ in $\Sigma(L)$,
obtained by taking branched double cover of $H$ along its two basepoints. 
\end{defn}

\begin{lem}
\label{lem:eqvtrans-HFK}Let $H$ be a link Heegaard diagram representing
a link in $S^{3}$ and write $B(H)=(\Sigma,\boldsymbol{\alpha},\boldsymbol{\beta},\mathbf{x},\mathbf{z})$.
Then for a generic 1-parameter family $J$ of $\mathbb{Z}_{2}$-equivariant
almost complex structures on $\text{Sym}^{g}(\Sigma\backslash(\mathbf{x}\cup\mathbf{z}))$,
every pair of generators $x,y\in\mathbb{T}_{\alpha}\cap\mathbb{T}_{\beta}$,
and every homotopy class $\phi\in\pi_{2}(x,y)$ which does not intersect
$\mathbf{x}$ and $\mathbf{z}$, the moduli spaces $\mathcal{M}(\phi;J)$
is transversely cut out.
\end{lem}

\begin{proof}
Since we are branching over $\mathbf{x}\cup\mathbf{z}$ and $\phi$
does not intersect with $\mathbf{x}$ and $\mathbf{z}$, condition
(EH-2) in \cite{eqv-Floer} is satisfied.
\end{proof}
\begin{lem}
\label{lem:eqvhfk-welldef}For any two weakly admissible link Heegaard
diagrams $H_{1},H_{2}$ reprsenting a link in $S^{3}$, we have 
\[
\widehat{HF}_{\mathbb{Z}_{2}}(B(H_{1}))\simeq\widehat{HF}(B(H_{2})).
\]
\end{lem}

\begin{proof}
We only have to consider the following three cases.
\begin{enumerate}
\item $H_{1}$ and $H_{2}$ are related by either an $\alpha$-isotopy or
a $\beta$-isotopy.
\item $H_{1}$ and $H_{2}$ are related by a handleslide.
\item $H_{2}$ is obtained by stabilizing $H_{1}$.
\end{enumerate}
Cases 1 and 2 can be dealt in the same way as in the case of $\widehat{HF}_{\mathbb{Z}_{2}}(\Sigma(K),z)$,
so we assume that $H_{2}$ is obtained by stabilizing $H_{1}$. By
Lemma \ref{lem:eqvtrans-HFK}, we have 
\begin{align*}
\widehat{CF}_{\mathbb{Z}_{2}}(B(H_{1})) & \simeq\text{RHom}_{\mathbb{F}_{2}[\mathbb{Z}_{2}]}(\widehat{CF}(B(H_{1})),\mathbb{F}_{2}),\\
\widehat{CF}_{\mathbb{Z}_{2}}(B(H_{2})) & \simeq\text{RHom}_{\mathbb{F}_{2}[\mathbb{Z}_{2}]}(\widehat{CF}(B(H_{2})),\mathbb{F}_{2}).
\end{align*}
But $B(H_{2})$ is the diagram obtained by performing a ($\mathbb{Z}_{2}$-invariant)
pair of stabilizations on $B(H_{1})$, so the stabilization map 
\[
\widehat{CF}(B(H_{1}))\rightarrow\widehat{CF}(B(H_{2}))
\]
 is a $\mathbb{Z}_{2}$-equivariant quasi-isomorphism. Therefore $\widehat{CF}_{\mathbb{Z}_{2}}(B(H_{1}))$
is quasi-isomorphic to $\widehat{CF}_{\mathbb{Z}_{2}}(B(H_{2}))$.
\end{proof}
Lemma \ref{lem:eqvhfk-welldef} suggests us to make the following
definition.
\begin{defn}
Let $L$ be a link in $S^{3}$ and $H$ be a Heegaard diagram representing
$L$. Then we define $\mathbb{Z}_{2}$-equivariant knot Floer cohomology
$\widehat{HFL}_{\mathbb{Z}_{2}}(\Sigma(L),L)$ as the isomorphism
class of the $\mathbb{F}_{2}[\theta]$-module $\widehat{HF}_{\mathbb{Z}_{2}}(B(H))$.
when $L$ is a knot, then we denote $\widehat{HFL}_{\mathbb{Z}_{2}}$
by $\widehat{HFK}_{\mathbb{Z}_{2}}$. 
\end{defn}

\begin{rem}
As in the case of $\widehat{HF}_{\mathbb{Z}_{2}}$ of based knots,
$\widehat{HFK}_{\mathbb{Z}_{2}}$ of links can also be computed in
a combinatorial way. The difference is that, for $\widehat{HFK}_{\mathbb{Z}_{2}}$,
it suffices to use any nice diagrams, instead of more restrictive
very nice diagrams, and the proof is much more simple. 

Given a Heegaard diagram $H_{0}$ representing a link $L$ in $S^{3}$,
we first apply Sarkar-Wang algorithm on $H_{0}$ to turn it into a
nice diagram $H$. Then $B(H)$ is also nice, since we are not throwing
away any basepoints, so $\widehat{CF}(B(H))$ can be combinatorially
computed. Then, by Lemma \ref{lem:eqvhfk-welldef}, we have 
\[
\widehat{HFK}_{\mathbb{Z}_{2}}(\Sigma(L),L)\simeq H^{\ast}(\text{RHom}_{\mathbb{F}_{2}[\mathbb{Z}_{2}]}(\widehat{CF}(B(H)),\mathbb{F}_{2})),
\]
 so we can compute $\widehat{HFK}_{\mathbb{Z}_{2}}$ combinatorially.
\end{rem}

The equivariant link Floer cohomology $\widehat{HFL}_{\mathbb{Z}_{2}}$
is, as its name suggests, a refinement of both the equivariant Floer
cohomology $\widehat{HF}_{\mathbb{Z}_{2}}$ of based knots in $S^{3}$
and the ordinary knot Floer cohomology $\widehat{HFL}$ of links in
$S^{3}$.
\begin{thm}
\label{thm:specseq1}For any link $L$ in $S^{3}$, there exists a
spectral sequence 
\[
\widehat{HFL}(\Sigma(L),L)\otimes\mathbb{F}_{2}[\theta]\Rightarrow\widehat{HFL}_{\mathbb{Z}_{2}}(\Sigma(L),L),
\]
 and the isomorphism class of its pages are isotopy invariants of
$L$.
\end{thm}

\begin{proof}
This is a direct consequence of Lemma \ref{lem:eqvtrans-HFK}.
\end{proof}
\begin{thm}
\label{thm:specseq2}For any knot $K$ in $S^{3}$, there exists a
spectral sequence 
\[
\widehat{HFK}_{\mathbb{Z}_{2}}(\Sigma(K),K)\Rightarrow\widehat{HF}_{\mathbb{Z}_{2}}(\Sigma(K),z)
\]
 for any $z\in K$, and the isomorphism class of its pages are isotopy
invariants of $K$.
\end{thm}

\begin{proof}
Let $H=(\Sigma,\boldsymbol{\alpha},\boldsymbol{\beta},z,w)$ be any
knot Heegaard diagram representing $K$. Then, as in the non-equivariant
case, the second basepoint $w$ induces a filtration on $\widehat{CF}_{\mathbb{Z}_{2}}(\hat{H})$,
and its isomorphism class as a filtered chain complex is an invariant
of $K$. The spectral sequence we get is of the form 
\[
\widehat{HF}_{\mathbb{Z}_{2}}(B(H))\Rightarrow\widehat{HF}_{\mathbb{Z}_{2}}(\hat{H}).
\]
 Now, $\widehat{HF}_{\mathbb{Z}_{2}}(B(H))\simeq\widehat{HFK}_{\mathbb{Z}_{2}}(\Sigma(K),K)$
by the definition of $\widehat{HFK}_{\mathbb{Z}_{2}}$, and by Theorem
\ref{thm:knotdiagramHF}, we have $\widehat{HF}_{\mathbb{Z}_{2}}(\hat{H})\simeq\widehat{HF}_{\mathbb{Z}_{2}}(\Sigma(K),z)$.
\end{proof}
Now we will prove a version of localization theorem for $\mathbb{Z}_{2}$-equivariant
link Floer cohomology. Recall from Proposition 6.3 of \cite{eqv-Floer}
that, for any based link $(L,z)$ in $S^{3}$, we have 
\[
\widehat{HF}_{\mathbb{Z}_{2}}(\Sigma(L),z)\otimes_{\mathbb{F}_{2}[\theta]}\mathbb{F}_{2}[\theta,\theta^{-1}]\simeq(\mathbb{F}_{2}^{2})^{|L|-1}\otimes_{\mathbb{F}_{2}}\mathbb{F}_{2}[\theta,\theta^{-1}],
\]
 where $|L|$ is the number of connected components of $L$. We will
prove that a similar statement is true for $\widehat{HFL}(\Sigma(L),L)$.
\begin{thm}
\label{thm:hfk-localization}For any link $L$ in $S^{3}$, we have
an isomorphism
\[
\widehat{HFL}_{\mathbb{Z}_{2}}(\Sigma(L),L)\otimes_{\mathbb{F}_{2}[\theta]}\mathbb{F}_{2}[\theta,\theta^{-1}]\simeq\widehat{HFL}_{\mathbb{Z}_{2}}(S^{3},L)\otimes_{\mathbb{F}_{2}}\mathbb{F}_{2}[\theta,\theta^{-1}].
\]
\end{thm}

\begin{proof}
Let $H=(\Sigma,\boldsymbol{\alpha},\boldsymbol{\beta},z,w)$ be a
weakly admissible Heegaard diagram representing $L$, and for each
positive integer $n$, denote by $H_{n}$ the Heegaard diagram obtained
by adding $n$ $\alpha$-curves, $n$ $\beta$-curves, and $2n$ basepoints
in a small neighborhood of $w$, as drawn in Figure \ref{fig08}.
The proof of Lemma \ref{lem:eqvtrans-HFK} also applies to the diagram
$H_{n}$, so if we denote the branched double cover of $H_{n}$ along
its basepoints by $B(H_{n})$, then we get an isomorphism 
\[
\widehat{CF}_{\mathbb{Z}_{2}}(B(H_{n}))\simeq\text{RHom}_{\mathbb{F}_{2}[\mathbb{Z}_{2}]}(\widehat{CF}(B(H_{n})),\mathbb{F}_{2}).
\]
 Since we have $\widehat{CF}(B(H_{n}))\simeq\widehat{CF}(B(H))\otimes C_{n}^{\ast}$
for a finite-dimensional chain complex $C_{n}^{\ast}$, where $\mathbb{Z}_{2}$
acts trivially on $C_{n}^{\ast}$, we have an isomorphism 
\[
\widehat{HF}_{\mathbb{Z}_{2}}(B(H_{n}))\simeq\widehat{HFL}_{\mathbb{Z}_{2}}(\Sigma(L),L)\otimes V_{n}
\]
 for a $\mathbb{F}_{2}$-vector space $V_{n}$ of dimension $2^{n}$.
Hence $V_{n}$ satisfies 
\[
\widehat{HF}(H_{n})\simeq\widehat{HFL}(S^{3},L)\otimes V_{n}.
\]

When $n$ is sufficiently large, there exists a multi-pointed Heegaard
diagram $H_{0}$, whose Heegaard surface is $S^{2}$, so that $H_{n}$
and $H_{0}$ are related by a sequence of isotopies, handleslides,
and (de)stabilizations, which would imply $\widehat{HF}_{\mathbb{Z}_{2}}(B(H_{0}))\simeq\widehat{HF}_{\mathbb{Z}_{2}}(B(H_{n}))$
by equivariant transversality, and $\widehat{HF}(H_{0})\simeq\widehat{HF}(H_{n})$.
Since $H_{0}$ is a Heegaard diagram drawn on $S^{2}$, it satisfies
condition (EH-1) in \cite{eqv-Floer}, so we have a localization isomorphism
\[
\widehat{HF}_{\mathbb{Z}_{2}}(B(H_{0}))\otimes_{\mathbb{F}_{2}[\theta]}\mathbb{F}_{2}[\theta,\theta^{-1}]\simeq\widehat{HF}(H_{0})\otimes_{\mathbb{F}_{2}}\mathbb{F}_{2}[\theta,\theta^{-1}].
\]
 We now consider the isomorphisms we have got. First, we have the
following isomorphism:
\[
\widehat{HF}(H_{0})\simeq\widehat{HF}(H_{n})\simeq\widehat{HFL}(S^{3},L)\otimes V_{n}.
\]
 Next, we have the following isomorphisms:
\begin{align*}
\widehat{HF}(H_{0})\otimes_{\mathbb{F}_{2}}\mathbb{F}_{2}[\theta,\theta]^{-1}\simeq & \widehat{HF}_{\mathbb{Z}_{2}}(B(H_{0}))\otimes_{\mathbb{F}_{2}[\theta]}\mathbb{F}_{2}[\theta,\theta^{-1}]\\
\simeq & \widehat{HF}_{\mathbb{Z}_{2}}(B(H_{n}))\otimes_{\mathbb{F}_{2}[\theta]}\mathbb{F}_{2}[\theta,\theta^{-1}]\\
\simeq & \widehat{HFL}_{\mathbb{Z}_{2}}(\Sigma(L),L)\otimes_{\mathbb{F}_{2}[\theta]}\mathbb{F}_{2}[\theta,\theta^{-1}]\otimes_{\mathbb{F}_{2}}V_{n}.
\end{align*}
 Hence we have the following isomorphism of $\mathbb{F}_{2}[\theta,\theta^{-1}]$-modules:
\[
\widehat{HFL}(S^{3},L)\otimes_{\mathbb{F}_{2}}\mathbb{F}_{2}[\theta,\theta^{-1}]\otimes V_{n}\simeq\widehat{HFL}_{\mathbb{Z}_{2}}(\Sigma(L),L)\otimes_{\mathbb{F}_{2}[\theta]}\mathbb{F}_{2}[\theta,\theta^{-1}]\otimes_{\mathbb{F}_{2}}V_{n}.
\]
 In other words, if we denote $\widehat{HFL}(S^{3},L)\otimes_{\mathbb{F}_{2}}\mathbb{F}_{2}[\theta,\theta^{-1}]$
by $M$, $\widehat{HFL}_{\mathbb{Z}_{2}}(\Sigma(L),L)\otimes_{\mathbb{F}_{2}[\theta]}\mathbb{F}_{2}[\theta,\theta^{-1}]$
by $N$, then we have 
\[
M^{2^{n}}\simeq N^{2^{n}}.
\]
 However, $M$ and $N$ are finitely generated $\mathbb{F}_{2}[\theta,\theta^{-1}]$-modules,
and since $\mathbb{F}_{2}[\theta]$ is a PID, its localization $\mathbb{F}_{2}[\theta,\theta^{-1}]$
is also a PID. Therefore, by the classfication of finitely generated
modules over a PID, we must have $M\simeq N$, i.e.
\[
\widehat{HFL}(S^{3},L)\otimes_{\mathbb{F}_{2}}\mathbb{F}_{2}[\theta,\theta^{-1}]\simeq\widehat{HFL}_{\mathbb{Z}_{2}}(\Sigma(L),L)\otimes_{\mathbb{F}_{2}[\theta]}\mathbb{F}_{2}[\theta,\theta^{-1}].
\]
\end{proof}
\begin{figure}
\resizebox{.5\textwidth}{!}{\includegraphics{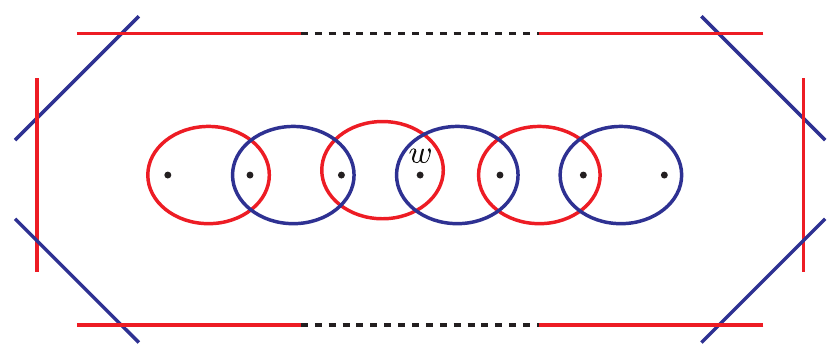}} \caption{\label{fig08}The diagram $H_{n}$. Here we have set $n=3$ for simplicity.}
\end{figure}

Theorem \ref{thm:hfk-localization} gives a new proof of Theorem 1.14
of \cite{eqv-Floer} as a direct corollary.
\begin{cor}
Given a link $L$ in $S^{3}$, there exists a spectral sequence 
\[
\widehat{HFL}(\Sigma(L),L)\otimes_{\mathbb{F}_{2}}\mathbb{F}_{2}[\theta,\theta^{-1}]\Rightarrow\widehat{HFL}_{\mathbb{Z}_{2}}(S^{3},L)\otimes_{\mathbb{F}_{2}[\theta]}\mathbb{F}_{2}[\theta,\theta^{-1}],
\]
 whose pages are isotopy invariants of $L$.
\end{cor}

\begin{proof}
Tensoring the spectral sequence of Theorem \ref{thm:specseq1} with
$\mathbb{F}_{2}[\theta,\theta^{-1}]$ gives a spectral sequence
\[
\widehat{HFL}(\Sigma(L),L)\otimes_{\mathbb{F}_{2}}\mathbb{F}_{2}[\theta,\theta^{-1}]\Rightarrow\widehat{HFL}_{\mathbb{Z}_{2}}(\Sigma(L),L)\otimes_{\mathbb{F}_{2}[\theta]}\mathbb{F}_{2}[\theta,\theta^{-1}],
\]
 and by Theorem \ref{thm:hfk-localization}, we have an isomorphism
\[
\widehat{HFL}_{\mathbb{Z}_{2}}(\Sigma(L),L)\otimes_{\mathbb{F}_{2}[\theta]}\mathbb{F}_{2}[\theta,\theta^{-1}]\simeq\widehat{HFL}(S^{3},L)\otimes_{\mathbb{F}_{2}}\mathbb{F}_{2}[\theta,\theta^{-1}].
\]
\end{proof}

\section{Application: A transverse invariant refining both the $c_{\mathbb{Z}_{2}}$
and the LOSS invariant}

Let $K\subset S^{3}$ be a transverse knot with respect to the standard
contact structure $\xi_{std}=\ker(dz-xdy)$. The author constructed
in \cite{Kang} an invariant $c_{\mathbb{Z}_{2}}(\xi_{K})$ as an
element of the module $\widehat{HF}_{\mathbb{Z}_{2}}(\Sigma(K))$,
which depends only on the transverse isotopy class of $K$. On the
other hand, we have the LOSS invariant $\hat{c}(K)$, defined in \cite{LOSS-inv},
which is an element of $\widehat{HFK}(S^{3},K)$, again depending
only on the transverse isotopy class of $K$. Therefore it is natural
to ask how those two invariants are related.

Recall that we have shown in the previous section that the equivariant
knot Floer cohomology $\widehat{HFK}_{\mathbb{Z}_{2}}(\Sigma(K),K)$
is a refinement of both $\widehat{HF}_{\mathbb{Z}_{2}}(\Sigma(K))$
and $\widehat{HFK}(S^{3},K)$. As a topological application of that
fact, we will show in this section that there exists a transverse
invariant of $K$, as an element of $\widehat{HFK}_{\mathbb{Z}_{2}}(\Sigma(K),K)$,
which is a simultaneous refinement of both the equivariant contact
class $c_{\mathbb{Z}_{2}}(\xi_{K})$ and the LOSS invariant $\hat{c}(K)$.

We first recall some facts regarding relations between braids and
transverse knots. When $(M,\xi)$ is a contact 3-manifold and $h\,:\,M\backslash B\rightarrow S^{1}$
is an open book supporting $\xi$. Then a transverse knot(link) $K\subset M$
is said to be in a \textbf{braid position} with respect to $B$ if
$h|_{K}$ is regular, and when $K$ is in a braid position, we call
it as a transverse braid. When $(M,\xi)$ is the standard contact
$S^{3}$, the theorem of Bennequin\cite{Bennequin} tells us that
every transverse knot is transversely isotopic to a transverse braid
along the $z$-axis. This fact can be directly generalized to the
most general case, when $(M,\xi)$ is any contact 3-manifold and $h$
is any open book; in fact, the theorem of Pavelescu\cite{Pavelescu}
tells us that any transverse knot in $M$ is transversely isotopic
to a braid position with respect to $h$, and two transverse braids,
braided with respect to $B$, are transversely isotopic if and only
if they are related by a sequence of braid isotopies, positive Markov
moves, and their inverses. This fact leads us to the following definition,
made first by Baldwin, Vela-Vick, and Vertesi in \cite{GRID=00003DLOSS}.
\begin{defn}
Given a contact 3-manifold $(M,\xi)$, an open book $h\,:\,M\backslash B\rightarrow S^{1}$
supporting $\xi$, and a transverse knot $K$ which is in a braid
position with respect to $B$, let $S$ be a page of $h$, $F=S\cap K$,
and $\phi$ be the monodromy of $h$ which fixes $F$ pointwise. Then
the isotopy class of the monodromy $\phi$ rel $F\cup\partial S$
uniquely determines the transverse isotopy class of $K$. So we say
that the (transverse isotopy class of) transverse knot $K$ is \textbf{encoded}
by the \textbf{pointed open book} $(S,F,\phi)$.
\end{defn}

Recall that, given a contact 3-manifold $(M,\xi)$, any two open books
of $M$ which support $\xi$ are related by a sequence of isotopies
and positive (de)stabilizations. Then, given any pointed open book
$(S,F,\phi)$ which encodes a transverse knot $K$ in $(M,\xi)$,
we can perform a positive stabilization with respect to any simple
arc $a\subset S\backslash F$ satisfying $\partial a\subset\partial S$.
Then the pointed open book $(S^{\prime},F^{\prime},\phi^{\prime})$
we obtain by the positive stabilization still encodes the same transverse
knot $K$; note that this is a special case of Corollary 2.5 of \cite{GRID=00003DLOSS}.

Using the facts we have stated above, we will give some useful observations
on transverse knots in the standard conatct $S^{3}$.
\begin{lem}
\label{lem:trans1}Let $K\subset(S^{3},\xi_{std})$ be a transverse
knot. Then there exists an open book decomposition of $S^{3}$ supporting
$\xi_{std}$, such that its binding is disjoint from $K$ and each
of its page intersects $K$ transversely at one point.
\end{lem}

\begin{proof}
Let $L$ be a Legendrian knot in $(S^{3},\xi_{std})$, whose positive
transverse pushoff is $K$. Choose a contact cell decomposition of
$(S^{3},\xi_{std})$, so that $L$ is contained in its 1-skeleton
$S$. Denote the 1-cells of $S$ which intersects $L$ by $L_{1},\cdots,L_{n}$,
and the open book induced by the cell decomposition as $h\,:\,S^{3}\backslash B\rightarrow S^{1}$,
where $B$ is the binding. Then a page $\Sigma$ of $h$ is drawn
in Figure \ref{fig09}. Note that $N(L)\cap\partial(h^{-1}(p))\subset K$,
where $N(L)$ is a small neighborhood of $L$; this is because $K$
is a positive pushoff of $L$.

Now perturb $K$ to a knot $K^{\prime}$, so that $h|_{K^{\prime}}$
is monotone but very close to zero in neighborhoods of $L_{1},\cdots,L_{n}$.
Then $K^{\prime}$ is clearly in a braid position with respect to
$B$. Also, since we can choose our perturbation to be sufficiently
$C^{1}$-small, $K^{\prime}$ is also a transverse knot. Therefore
$K^{\prime}$ is a 1-strand transverse braid along $B$, so applying
a self-contactomorphism of $(S^{3},\xi_{std})$ which takes $K$ to
$K^{\prime}$ gives a desired open book.
\end{proof}
\begin{figure}
\resizebox{.5\textwidth}{!}{\includegraphics{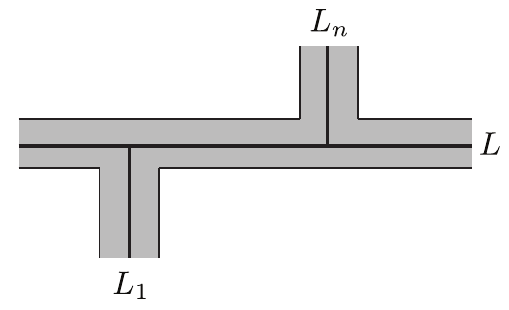}} \caption{\label{fig09}The page $\Sigma$ of the open book $h$, near the Legendrian
knot $L$.}
\end{figure}
\begin{lem}
\label{lem:trans2}Let $K\subset(S^{3},\xi_{std})$ be a transverse
knot and $(B_{1},h_{1}),(B_{2},h_{2})$ be open book decompositions
of $S^{3}$, supporting $\xi_{std}$, such that $K$ is transversely
isotopic to a 1-strand transverse braid along $B_{i}$, for each $i=1,2$.
Denote the induced (single-)pointed open books by $(S_{1},p_{1},\phi_{1})$
and $(S_{2},p_{2},\phi_{2})$, respectively. Then $(S_{1},p_{1},\phi_{1})$
and $(S_{2},p_{2},\phi_{2})$ are related by a sequence of isotopies
and positive (de)stabilizations.
\end{lem}

\begin{proof}
By applying a self-contactomorphism of $(S^{3},\xi_{std})$ to the
given open books, we may assume that $K$ intersects the pages of
$(B_{1},h_{1})$ and $(B_{2},h_{2})$ transversely at one point. Then,
by performing positive stabilizations to $(B_{1},h_{1})$ and $(B_{2},h_{2})$,
we may assume that they are induced by a contact cell decomposition
of $(S^{3},\xi_{std})$, so that $(B_{1},h_{1})$ is isotopic to $(B_{2},h_{2})$.
Also, by performing additional positive stabilizations, we may further
assume that, for a Legendrian knot $L$ which has $K$ as its transverse
push-off, the given contact cell decomposition contains $L$ in its
1-skeleton. Then, by the argument used in the proof of Lemma \ref{lem:trans1},
we deduce that $(S_{1},p_{1},\phi_{1})$ and $(S_{2},p_{2},\phi_{2})$
are isotopic.
\end{proof}
To sum up, we have proven the following proposition.
\begin{prop}
\label{prop:trans}Any transverse knot in $(S^{3},\xi_{std})$ is
encoded by a single-pointed open book, and any two single-pointed
open books encoding the same transverse knot in $(S^{3},\xi_{std})$
are related by a sequence of isotopies and positive (de)stabilizations.
\end{prop}

\begin{proof}
This is a direct consequence of Lemma \ref{lem:trans1} and Lemma
\ref{lem:trans2}.
\end{proof}
Now we will see how to translate a single-pointed open book, which
encodes a transverse knot $K$ in $(S^{3},\xi_{std})$, to an element
in $\widehat{HFK}_{\mathbb{Z}_{2}}(\Sigma(K),K)$. Given such a single-pointed
open book $(S,p,\phi)$, choose an arc-basis $\mathbf{a}=\{a_{1},\cdots,a_{n}\}$
of $S$, where the arcs in $\mathbf{a}$ do not contain $p$. Then
consider the following objects.
\begin{itemize}
\item $\Sigma=(S\times\{0,1\})/\sim$ where $(x,0)\sim(x,1)$ for any $x\in\partial S$
\item For each $i=1,\cdots,n$, $b_{i}$ is the arc in $S$ given by perturbing
$a_{i}$ slightly, so that $a_{i}$ intersects $b_{i}$ transversely
at one point $x_{i}$, where the intersection is positive.
\item $\boldsymbol{\alpha}=\{(a_{i}\times\{0\})\cup(a_{i}\times\{1\})\vert i=1,\cdots,n\}$
\item $\boldsymbol{\beta}=\{(b_{i}\times\{0\})\cup(\phi(b_{i})\times\{1\})\vert i=1,\cdots,n\}$
\item $z=(p,0)$, $w=(p,1)$ in $\Sigma$
\end{itemize}
Then the 5-tuple $H_{(S,p,\phi,\mathbf{a})}=(\Sigma,\boldsymbol{\alpha},\boldsymbol{\beta},z,w)$
is a knot Heegaard diagram representing $K$ in $S^{3}$, and its
branched double cover $\hat{H}$ along $\{z,w\}$ is a $\mathbb{Z}_{2}$-invariant
knot Heegaard diagram representing $K$ in $\Sigma(K)$, so that we
have 
\[
\widehat{HFK}_{\mathbb{Z}_{2}}(\Sigma(K),K)\simeq\widehat{HF}_{\mathbb{Z}_{2}}(\hat{H})\simeq H^{\ast}(\text{RHom}_{\mathbb{F}_{2}[\mathbb{Z}_{2}]}(\widehat{CF}(\hat{H}),\mathbb{F}_{2}))
\]
 by Lemma \ref{lem:eqvtrans-HFK}. So, if we denote the inverse image
of $\{x_{1},\cdots,x_{n}\}$ under the branched covering map by $\mathbf{x}_{(S,p,\phi,\mathbf{a})}$,
then since $\mathbf{x}_{(S,p,\phi,\mathbf{a})}$ is a $\mathbb{Z}_{2}$-invariant
cocycle in $\widehat{CF}(\hat{H})$, the element $\mathbf{x}_{(S,p,\phi,\mathbf{a})}\otimes1$
in $\text{RHom}_{\mathbb{F}_{2}[\mathbb{Z}_{2}]}(\widehat{CF}(\hat{H}),\mathbb{F}_{2})$
is a cocycle, which then corresponds to a cocycle in $\widehat{CFK}_{\mathbb{Z}_{2}}(\Sigma(K),K)$.
We will denote its cohomology class as $\hat{\mathcal{T}}_{\mathbb{Z}_{2}}(S,p,\phi,\mathbf{a})$.
Also, we will denote $\widehat{HF}_{\mathbb{Z}_{2}}(\hat{H})$ as
$\widehat{HFK}_{\mathbb{Z}_{2}}(S,p,\phi,\mathbf{a})$, and call the
tuple $(S,p,\phi,\mathbf{a})$ as a \textbf{single-pointed arc diagram}. 
\begin{lem}
\label{lem:isotopy-arcslide}Let $S$ be a compact oriented surface
with boundary, $p$ be a point in the interior of $S$, and $\mathbf{a}_{1},\mathbf{a}_{2}$
be arc bases of $S$ whose arcs do not contain $p$. Then $\mathbf{a}_{1}$
and $\mathbf{a}_{2}$ are related by a sequence of isotopies and arc-slides,
where the isotopies are performed outside $p$.
\end{lem}

\begin{proof}
We only have to prove that we can replace an isotopy through $p$
by a sequence of isotopies and handleslides performed in $S\backslash\{p\}$.
To prove this, choose an arc $a\in\mathbf{a}_{1}$ and suppose that
we want to isotope $a$ to another simple arc $a^{\prime}$, where
$p\notin a^{\prime}$, $a\cap a^{\prime}=\emptyset$, and $a,a^{\prime}$
cobound a strip $R\subset S$ where $p\in R$. Cutting $S$ along
the arcs in $\mathbf{a}_{1}\backslash\{a\}$ gives an annulus $A$,
where the two endpoints of $a$ lie in different components of $\partial A$
. Here, $A\backslash(a\cup a^{\prime})$ is a disjoint union of a
disk $D$ with the strip $R$. Then, by a sequence of isotopies and
arc-slides in $D$, we can replace $a$ by $a^{\prime}$, as in Figure
\ref{fig10}. Since $p\notin D$, we see that the isotopies we have
used are performed outside $p$.
\end{proof}
\begin{figure}
\resizebox{.7\textwidth}{!}{\includegraphics{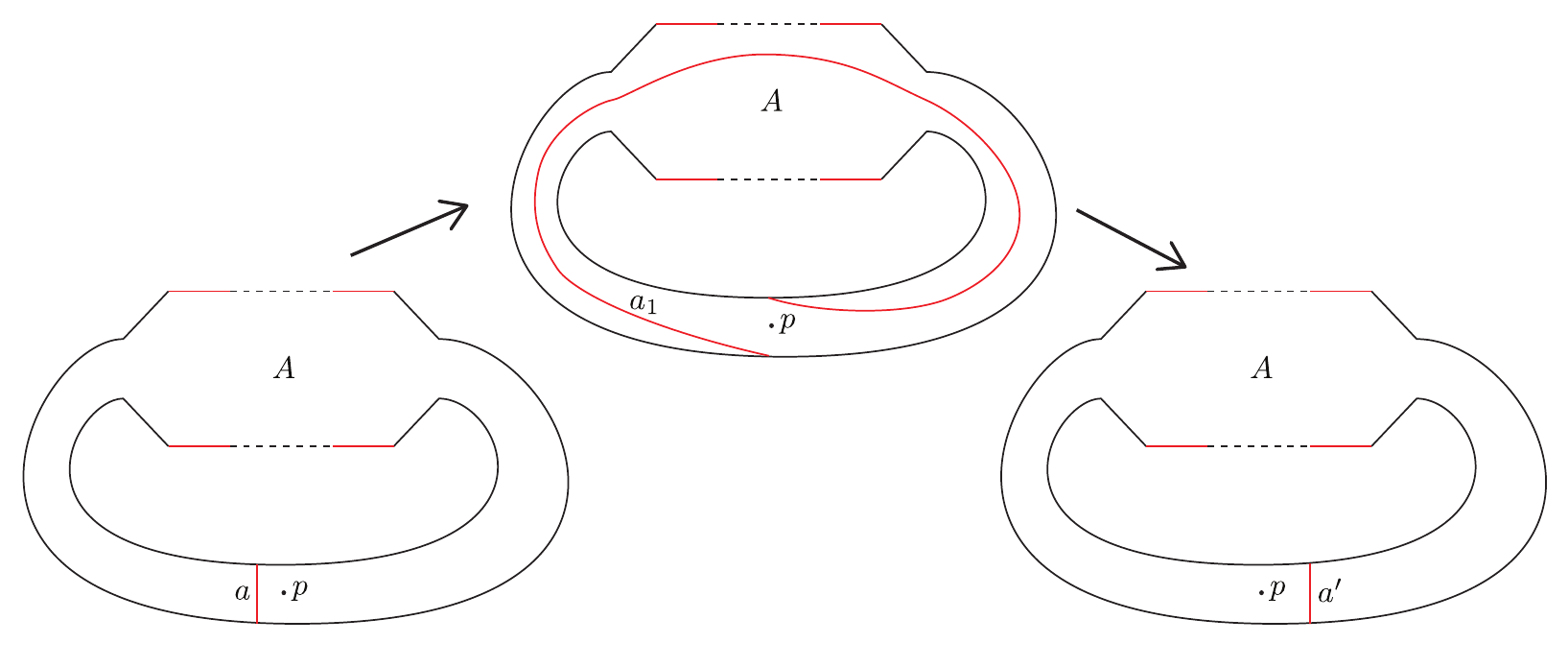}} \caption{\label{fig10}The arc $a_{1}$ is obtained from $a$ by a sequence
of arc-slides, and $a^{\prime}$ is obtained from $a_{1}$ again by
a sequence of arc-slides. During the process, the arc-slides we perform
do not pass through the basepoint $p$.}
\end{figure}
\begin{defn}
When two single-pointed arc diagrams $(S,p,\phi,\mathbf{a})$ and
$(S^{\prime},p^{\prime},\phi^{\prime},\mathbf{a}^{\prime})$ are related
by either an isotopy outside $p$, an arc-slide, or a stabilization,
we say that they are related by a \textbf{basic move}. Note that a
basic move induces a map 
\[
\widehat{HFK}_{\mathbb{Z}_{2}}(S^{\prime},p^{\prime},\phi^{\prime},\mathbf{a}^{\prime})\rightarrow\widehat{HFK}_{\mathbb{Z}_{2}}(S,p,\phi,\mathbf{a}).
\]
\end{defn}

\begin{lem}
\label{lem:loss-basicmove}Suppose that two single-pointed arc diagrams
$(S,p,\phi,\mathbf{a})$ and $(S^{\prime},p^{\prime},\phi^{\prime},\mathbf{a}^{\prime})$
are related by a basic move. Then the induced map 
\[
\widehat{HFK}_{\mathbb{Z}_{2}}(S^{\prime},p^{\prime},\phi^{\prime},\mathbf{a}^{\prime})\rightarrow\widehat{HFK}_{\mathbb{Z}_{2}}(S,p,\phi,\mathbf{a})
\]
 takes $\hat{\mathcal{T}}_{\mathbb{Z}_{2}}(S^{\prime},p^{\prime},\phi^{\prime},\mathbf{a}^{\prime})$
to $\hat{\mathcal{T}}_{\mathbb{Z}_{2}}(S,p,\phi,\mathbf{a})$.
\end{lem}

\begin{proof}
The branched double covers of $(S,p,\phi,\mathbf{a})$ and $(S^{\prime},p^{\prime},\phi^{\prime},\mathbf{a}^{\prime})$,
along $p$ and $p^{\prime}$, respectively, are related by two basic
moves, performed in a $\mathbb{Z}_{2}$-invariant way. Thus the conclusion
follows from the argument used to prove the invariance of the contact
classes of contact 3-manifolds; see Section 3 of \cite{HKM} for details.
\end{proof}
From the things we have proven so far, we see that we have found a
well-defined transverse invariant. Although we have not proven the
naturality of $\widehat{HFK}_{\mathbb{Z}_{2}}(\Sigma(K),K)$, we can
still define its well-defined element, which depends only on the transverse
isotopy class of $K$ in $S^{3}$.
\begin{thm}
\label{thm:eqvLOSS}The element $\hat{\mathcal{T}}_{\mathbb{Z}_{2}}(S,p,\phi,\mathbf{a})$
depends only on the transverse isotopy class of transverse knots which
$(S,p,\phi,\mathbf{a})$ encodes.
\end{thm}

\begin{proof}
This follows directly from Lemma \ref{lem:isotopy-arcslide} and Lemma
\ref{lem:loss-basicmove}.
\end{proof}
\begin{defn}
Given a (transverse isotopy class of) transverse knot $K$ in $(S^{3},\xi_{std})$,
we define $\hat{\mathcal{T}}_{\mathbb{Z}_{2}}(S,p,\phi,\mathbf{a})$,
for any single-pointed arc diagram $(S,p,\phi,\mathbf{a})$ encoding
$K$, as $\hat{\mathcal{T}}_{\mathbb{Z}_{2}}(K)$.
\end{defn}

Finally, we will see that $\hat{\mathcal{T}}_{\mathbb{Z}_{2}}(K)$
is a transverse invariant which is a refinement of both the hat-flavored
LOSS invariant $\hat{\mathcal{T}}(K)\in\widehat{HFK}(S^{3},K)$, and
the $\mathbb{Z}_{2}$-equivariant contact class $c_{\mathbb{Z}_{2}}(\xi_{K})\in\widehat{HF}_{\mathbb{Z}_{2}}(\Sigma(K))$,
defined in \cite{Kang}.
\begin{thm}
For any transverse knot $K$ in $(S^{3},\xi_{std})$, there exists
a localization isomorphism 
\[
\widehat{HFK}_{\mathbb{Z}_{2}}(\Sigma(K),K)\otimes_{\mathbb{F}_{2}[\theta]}\mathbb{F}_{2}[\theta,\theta^{-1}]\xrightarrow{\sim}\widehat{HFK}(S^{3},K)\otimes_{\mathbb{F}_{2}}\mathbb{F}_{2}[\theta,\theta^{-1}],
\]
 as defined in Theorem \ref{thm:hfk-localization}, which maps $\hat{\mathcal{T}}_{\mathbb{Z}_{2}}(K)\otimes1$
to $\hat{\mathcal{T}}(K)\otimes1$.
\end{thm}

\begin{proof}
Let $(S,p,\phi,\mathbf{a})$ be a single-pointed arc diagram encoding
$K$, so that $K$ is in a braid position with respect to the binding
$B=\partial S$. For a positive integer $n$, consider the transverse
braid $K_{n}$ given by applying postive Markov moves $n$ times to
$K$, and denote the $(n+1)$-pointed arc diagram encoding the braid
$K_{n}$ by $(S,\mathbf{p}_{n},\phi,\mathbf{a}_{n})$. Here, $\mathbf{p}_{n}=\{p,p_{1},\cdots,p_{n}\}$
where $p_{1},\cdots,p_{n}$ are pointes added by Markov moves, and
$\mathbf{a}_{n}$ is the arc basis given by adding to $\mathbf{a}$
small arcs connecting the points $p_{i}$ to the component of $\partial S$
where the $i$th Markov move was taken. Then, under the isomorphism
\[
\widehat{HFK}_{\mathbb{Z}_{2}}(S,p,\phi,\mathbf{a})\otimes_{\mathbb{F}_{2}}V_{n}\xrightarrow{\sim}\widehat{HFK}_{\mathbb{Z}_{2}}(S,\mathbf{p}_{n},\phi,\mathbf{a}_{n}),
\]
 the element $\hat{\mathcal{T}}_{\mathbb{Z}_{2}}(K)\otimes c$, where
$c$ denotes the top class in $V_{n}=(\mathbb{F}_{2}^{2})^{\otimes n}$,
is mapped to the cohomology class of the cocycle $\mathbf{x}_{H_{(S,\mathbf{p}_{n},\phi,\mathbf{a}_{n})}}$.

Now, when $n$ is sufficiently big, $(S,\mathbf{p}_{n},\phi,\mathbf{a}_{n})$
is related by a sequence of isotopies, arc-slides, and positive (de)stabilizations
to a pointed diagram of genus zero. So, by the argument used in the
proof of Theorem \ref{thm:hfk-localization}, together with Lemma
\ref{lem:loss-basicmove} applied to multi-pointed open books, we
have a localization isomorphism 
\[
\widehat{HFK}_{\mathbb{Z}_{2}}(S,\mathbf{p}_{n},\phi,\mathbf{a}_{n})\otimes_{\mathbb{F}_{2}[\theta]}\mathbb{F}_{2}[\theta,\theta^{-1}]\xrightarrow{\sim}\widehat{HFK}(S,\mathbf{p}_{n},\phi,\mathbf{a}_{n})\otimes_{\mathbb{F}_{2}}\mathbb{F}_{2}[\theta,\theta^{-1}],
\]
 which maps $\mathbf{x}_{H_{(S,\mathbf{p}_{n},\phi,\mathbf{a}_{n})}}\otimes1$
to $\hat{t}(K_{n})\otimes1$, where $\hat{t}(K_{n})$ is the image
of the BRAID invariant of the transverse braid $K_{n}$, defined in
\cite{GRID=00003DLOSS}, under the natural map from $HFK^{-}$ to
$\widehat{HFK}$. Also, by the LOSS=BRAID theorem (Theorem 5.1 of
\cite{GRID=00003DLOSS}) and the construction of localization map
in \cite{localization-map}, the isomorphism 
\[
\widehat{HFK}(S^{3},K)\otimes V_{n}\xrightarrow{\sim}\widehat{HFK}(S,\mathbf{p}_{n},\phi,\mathbf{a}_{n})
\]
 maps $\hat{\mathcal{T}}(K)\otimes c$ to $\hat{t}(K_{n})$. Therefore
there exists an isomorphism 
\[
\widehat{HFK}_{\mathbb{Z}_{2}}(\Sigma(K),K)\otimes_{\mathbb{F}_{2}[\theta]}\mathbb{F}_{2}[\theta,\theta^{-1}]\xrightarrow{\sim}\widehat{HFK}(S^{3},K)\otimes_{\mathbb{F}_{2}}\mathbb{F}_{2}[\theta,\theta^{-1}]
\]
 which maps $\hat{\mathcal{T}}_{\mathbb{Z}_{2}}(K)\otimes1$ to $\hat{\mathcal{T}}(K)\otimes1$
.
\end{proof}
Proving that $\hat{\mathcal{T}}_{\mathbb{Z}_{2}}(K)$ is a refinement
of $c_{\mathbb{Z}_{2}}(K)$ is a bit more difficult, as it needs an
extension of the definition of $c_{\mathbb{Z}_{2}}(\xi_{K})$. Recall
that, for a transverse knot $K$ in $(S^{3},\xi_{std})$, the $\mathbb{Z}_{2}$-equivariant
contact class $c_{\mathbb{Z}_{2}}(\xi_{K})$ is defined as follows.
\begin{itemize}
\item Choose a multi-pointed genus zero open book $(D,\mathbf{p}=\{p_{1},\cdots,p_{n}\},\phi)$
which encodes $K$.
\item Choose a system of pairwise disjoint arcs $\mathbf{a}=\{a_{1},\cdots,a_{n}\}$
such that, for each $i=2,\cdots,n$, $a_{i}$ starts from $p_{i}$
and ends at a point in $\partial D$.
\item Taking the branched double cover of $(D,\phi,\mathbf{a}\backslash\{a_{1}\})$,
branched along $\mathbf{p}$, gives a $\mathbb{Z}_{2}$-invariant
arc diagram $(\Sigma,\phi,\tilde{\mathbf{a}})$, where the point $p_{1}$
now works as a basepoint.
\item Applying Honda-Kazez-Matic construction of contact classes, as given
in the paper \cite{HKM}, gives an element $EH_{\mathbb{Z}_{2}}(\xi_{K})\in\widehat{CF}_{\mathbb{Z}_{2}}(\Sigma,\phi,\tilde{\mathbf{a}})$,
which turns out to be a cocycle which depends only on the transverse
isotopy class of $K$; see \cite{Kang} for details.
\item The cohomology class of $EH_{\mathbb{Z}_{2}}(\xi_{K})$ is denoted
as $c_{\mathbb{Z}_{2}}(\xi_{K})$.
\end{itemize}
We will now extend this definition to multi-pointed open books of
arbitrary genera.
\begin{defn}
Let $S$ be a compact oriented surface with boundary and $\mathbf{p}\subset\text{int}(S)$
be a finite subset, say $\mathbf{p}=\{p_{1},\cdots,p_{n}\}$. Then
an \textbf{extended arc basis} of the multi-pointed surface $(S,\mathbf{p})$
is a set $\mathbf{a}=\{a_{1}^{w},\cdots,a_{m}^{w},a_{1}^{h},\cdots,a_{n}^{h}\}$
of pairwise disjoint simple arcs, which satisfies the following conditions.
Here, the arcs $a_{i}^{w}$ are called \textbf{whole arcs}, and $a_{j}^{h}$
are called \textbf{half arcs}.
\begin{itemize}
\item $\{a_{1}^{w},\cdots,a_{m}^{w}\}$ is an arc basis of $S$.
\item For each $i=1,\cdots,n$, the arc $a_{i}^{h}$ starts from $p_{i}$
and ends at a point in $\partial S$.
\end{itemize}
\end{defn}

Given a multi-pointed open book $(S,\mathbf{p},\phi)$ encoding a
transverse knot $K$ in $(S^{3},\xi_{std})$, choose an extended arc
basis $\mathbf{a}$ of $(S,\mathbf{p})$ together with a distinguished
element $p_{1}\in\mathbf{p}$. Then, as in the genus zero case, taking
the branched double cover of $(S,\phi,\mathbf{a}\backslash\{a_{1}^{h}\})$
along $\mathbf{p}$, where $a_{1}^{h}$ is the half arc in $\mathbf{a}$
which contains $p_{1}$, gives a $\mathbb{Z}_{2}$-invariant arc diagram
$(\Sigma,\phi,\tilde{\mathbf{a}})$, where $p_{1}$ is now a basepoint
in $\Sigma\backslash\cup\tilde{\mathbf{a}}$. Applying Honda-Kazez-Matic
construction then gives a canonical element $EH_{\mathbb{Z}_{2}}(S,\mathbf{p},\phi,\mathbf{a},p_{1})$. 
\begin{defn}
The argument used to prove that $EH_{\mathbb{Z}_{2}}(\xi_{K})$ is
a cocycle can be directly applied to show that $EH_{\mathbb{Z}_{2}}(S,\mathbf{p},\phi,\mathbf{a},p_{1})$
is also a cocycle in $\widehat{CF}_{\mathbb{Z}_{2}}(\Sigma,\phi,\tilde{\mathbf{a}},p_{1})$.
So we denote its cohomology class as $c_{\mathbb{Z}_{2}}(S,\mathbf{p},\phi,\mathbf{a},p_{1})$.
\end{defn}

When $(S,\mathbf{p},\phi)$ is a multi-pointed open book, $p_{1}\in\mathbf{p}$
is a distinguished point, and $\mathbf{a}$ is an extended arc basis
of $(S,\mathbf{p})$, then we call the tuple $(S,\mathbf{p},\phi,\mathbf{a})$
as an \textbf{extended arc diagram}. The difference between using
genus zero extended arc diagrams and using diagrams of arbitrary genera
is that we now have four possible types of arc-slides. As usual, we
call isotopies and arc-slides as basic moves. Note that the HKM construction
applied to an extended arc diagram gives an extended bridge diagram,
and arc-slides of type I/II/III/IV correspond to the same types of
handleslides.

\subsection*{Isotopy}

We can perform an isotopy to arcs in an extended arc basis $\mathbf{a}$
of a multi-pointed surface $(S,\mathbf{p})$. Here, the isotopies
must not pass through points in $\mathbf{p}$.

\subsection*{Arc-slide of type I}

We can perform an (ordinary) arc-slide of an whole arc along another
whole arc in an extended arc basis $\mathbf{a}$, outside the basepoint
$p_{1}$ and the half-arcs in $\mathbf{a}$.

\subsection*{Arc-slide of type II}

Given a whole arc $a^{w}$ and a half arc $a^{h}$ in an extended
arc basis $\mathbf{a}$, we can replace $a^{w}$ by another whole
arc $a_{1}^{w}$ if $a_{1}^{w}$ does not intersect the arcs in $\mathbf{a}\backslash\{a^{w}\}$,
$a_{1}^{w}$ intersects $a^{w}$ transversely at one point, and $S\backslash(a^{w}\cup a_{1}^{w})$
contains a triangle component $T$ such that $a^{h}\subset T$.

\subsection*{Arc-slide of type III}

Given a half arc $a^{h}$ and a whole arc $a^{w}$ in an extended
arc basis $\mathbf{a}$, we can replace $a^{h}$ by another arc $a_{1}^{h}$
if $a^{h}\cap a_{1}^{h}$ is a point in $\mathbf{p}$, the interior
of $a_{1}^{h}$ does not intersect the arcs in $\mathbf{a}$, and
the arcs $a^{h},a^{w},a_{1}^{h}$ bound a strip in $S$.

\subsection*{Arc-slide of type IV}

Given two half arcs $a_{1}^{h},a_{2}^{h}$ in an extended arc basis
$\mathbf{a}$, we can replace $a_{1}^{h}$ by another arc $a_{1^{\prime}}^{h}$
if $a_{1}^{h}\cap a_{1^{\prime}}^{h}$ is a point in $\mathbf{p}$,
the interior of $a_{1^{\prime}}^{h}$ does not intersect the arcs
in $\mathbf{a}$, and $S\backslash(a_{1}^{h}\cup a_{1^{\prime}}^{h})$
contains a triangle component $T$ satisfying $a_{2}^{h}\subset T$.
\begin{lem}
\label{lem:exarcbasis}Any two arc bases of a multi-pointed surface
$(S,\mathbf{p})$ are related by a sequence of isotopies and arc-slides.
\end{lem}

\begin{proof}
This can be seen easily by combining the argument used in the proof
of Lemma \ref{lem:isotopy-arcslide} together with the proof of Proposition
5.3 in \cite{Kang}.
\end{proof}
\begin{prop}
\label{prop:inv-cz2}Let $(S,\mathbf{p},\phi,\mathbf{a},p_{1})$ be
an extended arc diagram, where $(S,\mathbf{p},\phi)$ encodes a transverse
knot $K$ in $(S^{3},\xi_{std})$. Then $c_{\mathbb{Z}_{2}}(S,\mathbf{p},\phi,\mathbf{a},p_{1})$
depends only on the transverse isotopy class of $K$.
\end{prop}

\begin{proof}
The proof is essentially the same as in the proof of invariance of
$c_{\mathbb{Z}_{2}}(\xi_{K})$ for genus zero open books. See Section
5 of \cite{Kang} for details.
\end{proof}
Finally, we can prove that for transverse knots $K$ in $(S^{3},\xi_{std})$,
$\hat{\mathcal{T}}_{\mathbb{Z}_{2}}(K)$ is a refinement of $c_{\mathbb{Z}_{2}}(\xi_{K})$.
\begin{thm}
For any transverse knot $K$ in $(S^{3},\xi_{std})$, the natural
map 
\[
\widehat{HFK}_{\mathbb{Z}_{2}}(\Sigma(K),K)\rightarrow\widehat{HF}_{\mathbb{Z}_{2}}(\Sigma(K))
\]
 maps $\hat{\mathcal{T}}_{\mathbb{Z}_{2}}(K)$ to $c_{\mathbb{Z}_{2}}(\xi_{K})$.
\end{thm}

\begin{proof}
Let $(S,p,\phi,\mathbf{a})$ be a single\textendash pointed arc diagram,
where $(S,p,\phi)$ encodes $K$. Choose a simple arc $a^{h}$ on
$S\backslash\cup\mathbf{a}$ which starts from $p$ and ends at a
point in $\partial S$. Then the image of $\mathbf{x}_{H_{(S,p,\phi,\mathbf{a})}}$
in the chain level is given by $EH_{\mathbb{Z}_{2}}(S,p,\phi,\mathbf{a}\cup\{a^{h}\},p)$.
Since $\hat{\mathcal{T}}_{\mathbb{Z}_{2}}(K)$ is the cohomology class
of $\mathbf{x}_{H_{(S,p,\phi,\mathbf{a})}}$, and the cohomology class
of $EH_{\mathbb{Z}_{2}}(S,p,\phi,\mathbf{a}\cup\{a^{h}\},p)$ is equal
to $c_{\mathbb{Z}_{2}}(\xi_{K})$ by \ref{prop:inv-cz2}, we see that
$\hat{\mathcal{T}}_{\mathbb{Z}_{2}}(K)$ is mapped to $c_{\mathbb{Z}_{2}}(\xi_{K})$.
\end{proof}

\end{document}